\documentclass[12pt,twoside]{amsart}
\input xy
\xyoption{all}
\usepackage{amssymb,amscd}
\usepackage{mathrsfs}
\usepackage{array,float}

\theoremstyle{plain}
\newtheorem{thm}{Theorem}[section]
\newtheorem{lem}[thm]{Lemma}
\newtheorem{pro}[thm]{Proposition}
\newtheorem{cor}[thm]{Corollary}
\newtheorem*{theorema}{Theorem B}
\newtheorem*{theoremb}{Theorem A}
\newtheorem*{theoremc}{Theorem C}
\newtheorem*{theoremd}{Theorem D}

\theoremstyle{remark}\theoremstyle{remark}
\newtheorem*{rem*}{Remark}

\theoremstyle{definition}
\newtheorem{definition}{Definition}

\numberwithin{equation}{section}

\newcommand{\Z}{\mathbb{Z}}

\def\G{\mathcal{G}}

\newcommand{\ackn}{  \noindent{\sc Acknowledgement }\hspace{5pt} }

\begin{document}

\title[Subgroup properties of pro-$p$ extensions of centralizers]{Subgroup properties of pro-$p$ extensions of centralizers}

\author{Ilir Snopce}
\address{Universidade Federal do Rio de Janeiro\\
  Instituto de Matem\'atica \\
  20785-050 Rio de Janeiro, RJ, Brasil }
\email{ilir@im.ufrj.br}

\thanks{This research was partially supported by CNPq}

\author{Pavel A. Zalesskii}
\address{Universidade de Bras\'ilia \\
  Departamento de Matem\'atica\\
  70910-900 Bras\'ilia, DF, Brasil }
\email{pz@mat.unb.br}

\begin{abstract}   We prove that a finitely generated pro-$p$ group acting on a pro-$p$ tree $T$ with procyclic edge
stabilizers is the fundamental pro-$p$ group of a finite graph of
pro-$p$ groups with edge and vertex groups being stabilizers of
certain vertices and edges of $T$ respectively, in the following
two situations: 1) the action is {\it $n$-acylindrical}, i.e., any
non-identity element fixes not more than $n$ edges; 2)  the group
$G$ is generated by its vertex stabilizers.  This theorem is
applied to obtain several results about pro-$p$ groups from the
class $\mathcal{L}$ defined and studied in \cite{Kochloukova2} as
pro-$p$ analogues of limit groups.  We prove that every pro-$p$
group $G$ from the class $\mathcal{L}$  is the fundamental pro-$p$ group of a finite graph of
pro-$p$ groups with infinite procyclic or trivial edge groups and
finitely generated vertex groups; moreover, all
 non-abelian  vertex groups are from the class
$\mathcal{L}$ of lower level than $G$ with respect to the natural hierarchy. This allows us to
give an affirmative answer to questions 9.1 and 9.3 in \cite{Kochloukova2}. Namely, we prove that a group $G$ from the class
 $\mathcal{L}$  has Euler-Poincar\'e characteristic zero if and only if it is abelian, and if every abelian pro-$p$ subgroup
 of $G$ is procyclic and $G$ itself is not procyclic, then $\textrm{def}(G)\geq 2$. Moreover, we prove that $G$ satisfies the Greenberg-Stallings property and any finitely generated
non-abelian subgroup
of $G$ has finite index in its commensurator.\\
  \indent  We also show that all non-solvable Demushkin groups satisfy the Greenberg-Stallings property and each of their finitely generated non-trivial subgroups has finite index in its commensurator.
\end{abstract}


\maketitle

\section{Introduction}

The main structure theorem of the Bass-Serre theory states that a
group $G$ acting on a tree $T$ is the fundamental group of a graph
of groups whose vertex and edge groups are the stabilizers of
certain vertices and edges of $T$. This means that $G$ can be
described by taking iterated amalgamated free products and HNN
extensions. The analogue of the structure theorem in the pro-$p$
case does not hold in general \cite{HZ:10}. Nevertheless, it was
proved in \cite{HZZ} that every finitely generated infinite
pro-$p$ group that acts virtually freely on some pro-$p$ tree $D$
is isomorphic to the fundamental pro-$p$ group of a finite graph
of finite $p$-groups whose edge and vertex groups are isomorphic
to the stabilizers of some edges and vertices of $D$. The first
objective of our paper is to prove  that such a pro-$p$ version of
the Bass-Serre theory structure theorem  holds for finitely
generated  pro-$p$ groups acting on a pro-$p$ tree with cyclic
edge stabilizers in any of the following two situations:

1) the action is {\it $n$-acylindrical}, i.e.,  any non-identity
element fixes not more than $n$ consecutive edges;

2)  the group $G$ is generated by its vertex stabilizers.

\begin{theoremb}\label{StructureThmGen}
Let $G$ be a finitely generated pro-$p$ group acting on a pro-$p$
tree $T$ with procyclic edge stabilizers. Suppose that either the
action is $n$-acylindrical or $G$ is generated by its vertex
stabilizers. Then $G$ is the
 fundamental pro-$p$ group of a finite graph of
pro-$p$ groups $(\G, \Gamma)$ with procyclic edge groups and
finitely generated vertex groups. Moreover, the vertex and edge
groups of $(\G, \Gamma)$ are  stabilizers of certain vertices and
edges of $T$ respectively, and stabilizers of vertices and edges
of $T$ in $G$ are conjugate to subgroups of vertex and edge groups
of $(\G, \Gamma)$ respectively.
\end{theoremb}

The original motivation for this study was an attempt to
investigate further the  pro-$p$ analogues of abstract limit
groups  defined and studied by Kochloukova and the second author
in \cite{Kochloukova2}.

Limit groups have been studied extensively over the last ten years and they played a crucial role in the solution of the
Tarski problem [12-14, 27-32]. The name \emph{limit group} was introduced by Sela. There are different equivalent
definitions for these groups. The class of limit groups coincides with the class of fully residually free  groups; under
this name they were studied by Remeslennikov, Kharlampovich and Myasnikov. One can also define  limit groups as finitely
generated subgroups of groups obtained from free groups of finite rank by finitely many extensions of centralizers.
Starting from this definition,  a special class $\mathcal{L}$ of pro-$p$ groups  (pro-$p$ analogues of  limit groups) was
introduced in \cite{Kochloukova2}. The class $\mathcal{L}$ consists of all finitely generated subgroups of pro-$p$ groups
obtained from free pro-$p$ groups of finite rank by finitely many extensions of centralizers. In \cite{Kochloukova2} it
was shown that many properties that hold for  limit groups are also satisfied by the pro-$p$ groups from the class
$\mathcal{L}$. In the present paper we study further the group theoretic structure properties of the pro-$p$ groups from
the class $\mathcal{L}$ and prove some other results that are known to hold in the abstract case.

\medskip
It is well known that a freely-indecomposable limit group of height $h\geq 1$ is the fundamental group of a finite graph of
groups that has infinite cyclic edge groups and has a vertex group that is a non-abelian limit group of height $\leq h-1$;
for example, see Proposition 2.1 in \cite{Bridson}. This fact allows one to prove many interesting properties for limit
groups using induction arguments. The main theorem of this paper is  an analogue of this result for
pro-$p$ groups from the
class $\mathcal{L}$.

\begin{theorema}\label{StructureThm}
Let $G$ be a pro-$p$ group from the class $\mathcal{L}$. If $G$ has weight $n\geq 1$, then it is the fundamental pro-$p$ group of a finite graph of
pro-$p$ groups that has infinite procyclic or trivial edge groups and finitely generated vertex groups. Moreover, if $G$ is non-abelian, then it has at least one vertex group that is a non-abelian pro-$p$ group and all the non-abelian vertex groups of $G$ are pro-$p$ groups from the class $\mathcal{L}$ of weight $\leq n-1$.
\end{theorema}

Case 1) of Theorem A is the key ingredient of the proof of Theorem B.

 \smallskip

 Theorem B has some interesting consequences. In \cite{Kochloukova1} Kochloukova proved that any limit group $G$ has
non-positive Euler-Poincar\'e characteristic $\chi(G)$ and that
$\chi(G)=0$ if and only if $G$ is abelian. Inspired from this
result, in \cite{Kochloukova2},  Kochloukova and the second author
proved that any pro-$p$ group $G$ from the class $\mathcal{L}$ has
a non-positive Euler-Poincar\'e characteristic and raised the
question whether it is true that $\chi(G)=0$ if and only if $G$ is
abelian (see question 9.3 in \cite{Kochloukova2}). We use Theorem
B to give an affirmative answer to this question. In the same
paper, Kochloukova and the second author noted that if $G$ is a
limit group such that every abelian subgroup of $G$ is cyclic and
$G$ itself is not cyclic then  the deficiency $\textrm{def}(G)\geq
2$, and they raised the question whether the analogue of this
result is also true for pro-$p$ groups from the class
$\mathcal{L}$ (see question 9.1 in \cite{Kochloukova2}). We use
Theorem B once more to give a positive answer to this question.

\smallskip

In \cite{Stallings}, based on results of Greenberg \cite{Greenberg} , Stallings proved that if $G$ is a  free group and
$H$ and $K$ are finitely
generated subgroups of $G$ with the property that $H\cap K$ has finite index in both $H$ and $K$, then $H\cap K$ has finite
index in $\langle H, K \rangle$, where $\langle H, K \rangle$ denotes the subgroup of $G$ generated by $H$ and $K$.
Nowadays this property is known as Greenberg-Stallings property. Kapovich \cite{Kapovich} proved that finitely generated
word-hyperbolic fully residually free
groups satisfy the Greenberg-Stallings property. Nikolaev and Serbin extended it to all limit groups \cite{Serbin}. In
this paper we prove that all pro-$p$
groups from the class $\mathcal{L}$ satisfy this property.

\smallskip

In \cite{Rosset} Rosset proved that every finitely generated subgroup $H$ of a free group $F$ has a ``root'': a subgroup
$K$ of $F$ that contains $H$ with
$|K:H|$ finite and which contains every subgroup $U$ of $F$ that contains $H$ with $|U:H|$ finite. We extend the result of
Rosset to the class of all   limit groups. We also prove the existence of the root for finitely generated closed subgroups
of  pro-$p$ groups from the class $\mathcal{L}$.
This allows us to show that every non-abelian finitely generated closed subgroup $H$ of a pro-$p$ group $G$ from the
class $\mathcal{L}$ has finite index
 in its commensurator $\textrm{Comm}_G(H)$. This property is also satisfied by  abstract limit groups \cite{Serbin}.

\smallskip

We list our results for the pro-$p$ analogues of limit groups in the following.

\begin{theoremc}\label{Finiteness}
Let $G$ be a pro-$p$ group from the class $\mathcal{L}$. Then
\begin{itemize}
\item[(1)] The group $G$ has a non-positive Euler-Poincar\'e characteristic. Moreover $\chi(G)=0$ if and only if $G$ is
abelian;
\item[(2)] If every abelian pro-$p$ subgroup of $G$ is procyclic and $G$ itself is not procyclic, then $\textrm{def}(G)\geq 2$;
\item[(3)] If every abelian pro-$p$ subgroup of $G$ is procyclic and $G$ itself is not procyclic, then $G$ has exponential subgroup growth;
\item[(4)] There are only finitely many conjugacy classes of non-procyclic maximal abelian subgroups of $G$;
\item[(5)] [Greenberg-Stallings Property] If $H$ and $K$ are finitely generated subgroups of $G$ with the property that $H\cap K$ has finite index in both $H$ and $K$, then $H\cap K$ has finite index in $\langle H, K \rangle$;
\item[(6)] If $H$ is a finitely generated subgroup of $G$, then $H$ has a root in $G$;
\item[(7)] If $H$ is a finitely generated non-abelian subgroup of $G$, then $|\textrm{Comm}_G(H):H|<\infty$.
\end{itemize}
\end{theoremc}

By Corollary 5.5 in \cite{Kochloukova2}, we know that a solvable Demushkin group belongs to the class $\mathcal{L}$ if and
only if it is abelian. It is not clear which non-solvable Demushkin groups belong to the class $\mathcal{L}$.
In \cite{Kochloukova2} it was shown that if $G$ is a Demushkin group with the invariant $q=\infty$ and $d(G)$ divisible
by 4, then $G\in \mathcal{L}$; in the remaining cases it is not known whether $G\in \mathcal{L}$. Anyway, we show that
parts $(5)$, $(6)$ and $(7)$ of the above theorem also hold for any non-solvable Demushkin group $G$. Indeed, we study a
more general family of groups that includes finitely generated free pro-$p$ groups and Demushkin groups, and prove the
following.

\begin{theoremd}\label{Dem}
Let $G$ be a pro-$p$ group with the property that all infinite index finitely generated subgroups of $G$ are free pro-$p$. Suppose that $G$ is finitely presented and has an open subgroup of deficiency greater than 1.  Then
\begin{itemize}
\item[(1)] If $H$ is a finitely generated subgroup of $G$ that contains a non-trivial normal subgroup of $G$, then $H$ has finite index in $G$;
\item[(2)] [Greenberg-Stallings Property] If $H$ and $K$ are finitely generated subgroups of $G$ with the property that $H\cap K$ has finite index in both $H$ and $K$, then $H\cap K$ has finite index in $\langle H, K \rangle$;
\item[(3)] If $H$ is a finitely generated subgroup of $G$, then $H$ has a root in $G$;
\item[(4)] Suppose in addition that all infinite index subgroups of $G$ are free pro-$p$ groups. Then $|\textrm{Comm}_G(H):H|<\infty$ for any non-trivial finitely generated subgroup $H$ of $G$.
\end{itemize}
\end{theoremd}

\smallskip

We note that we can not use in our proofs standard combinatorial methods as in the abstract case because not all elements
of  pro-$p$ groups can be expressed as finite words of generators.

\medskip

\noindent \textit{Organization.}  We prove Theorem A in section 2.
In section 3 we prove Theorem B and parts $(1)$, $(2)$, $(3)$ and $(4)$
of Theorem C. Parts $(5)$, $(6)$ and $(7)$ of  Theorem C are
proved in section 4.  Theorem D is proved in section 5; as an
immediate consequence we get our results for Demushkin groups. In
section 6 we note that every finitely generated subgroup of an
abstract limit group has a root.

\medskip

\noindent \textit{Notation.}  Throughout the paper $p$ denotes a
prime. The $p$-adic integers are denoted
by $\Z_p$. When $G$ is a topological group, then subgroups of $G$ are tacitly taken to be closed, unless otherwise stated; also $d(G)$ tacitly refers to the minimal number of topological generators of $G$. Moreover, homomorphisms between topological groups are tacitly taken to be continuous.
For a pro-$p$ group $G$ acting continuously on a pro-$p$ tree $T$ we define $\tilde{G}:=\langle G_x \mid x\in T \rangle$, where $G_x$ is the stabilizer of the point $x$.

\section{The decomposition theorem for pro-$p$ groups acting on a pro-$p$ tree $T$ with procyclic edge stabilizers}

In this section we prove Theorem A, stated in the introduction. We
start with some definitions, following \cite{Ribes2}. A
\emph{profinite graph} is a triple $(\Gamma, d_0, d_1)$, where
$\Gamma$ is a boolean space and $d_0, d_1: \Gamma \to \Gamma $ are
continuous maps such that $d_id_j=d_j$ for $i, j \in \{0, 1 \}$.
The elements of $V(\Gamma):=d_0(G)\cup d_1(G)$ are called the
\emph{vertices} of $\Gamma$ and the elements of
$E(\Gamma):=\Gamma-V(\Gamma)$ are called the \emph{edges} of
$\Gamma$. If $e\in E(\Gamma)$, then $d_0(e)$ and $d_1(e)$ are
called the initial and terminal vertices of $e$. If there is no
confusion, one can just write $\Gamma$ instead of $(\Gamma, d_0,
d_1)$.

Let $(E^*(\Gamma), *)=(\Gamma/V(\Gamma), *)$ be a pointed profinite quotient space with $V(\Gamma)$ as a distinguished point, and let $\mathbb{F}_p[[E^*(\Gamma), *]]$ and $\mathbb{F}_p[[V(\Gamma)]]$ be respectively the free profinite $\mathbb{F}_p$-modules over the pointed profinite space $(E^*(\Gamma), *)$ and over the profinite space $V(\Gamma)$ (cf. \cite{Ribes1}). Let the maps $\delta: \mathbb{F}_p[[E^*(\Gamma), *]] \to \mathbb{F}_p[[V(\Gamma)]] $ and $\epsilon : \mathbb{F}_p[[V(\Gamma)]] \to \mathbb{F}_p$ be defined respectively by $\delta(e)=d_1(e)-d_0(e)$ for all $e\in E^*(\Gamma)$ and $\epsilon(v)=1$ for all $v\in V(\Gamma)$. Then we have the following complex of free profinite $\mathbb{F}_p$-modules
\begin{equation*}
    \begin{CD}
      0 @>>> \mathbb{F}_p[[E^*(\Gamma), *]] @>\delta>> \mathbb{F}_p[[V(\Gamma)]] @>\epsilon>>
      \mathbb{F}_p @>>> 0.
    \end{CD}
  \end{equation*}

 We say that the profinite graph $\Gamma$ is a \emph{pro-$p$ tree} if the above sequence is exact. If $T$ is a pro-$p$ tree, then we say that a pro-$p$ group $G$ acts on $T$ if it acts continuously on $T$ and the action commutes with $d_0$ and $d_1$. For $t\in V(T)\cup E(T)$ we denote by $G_t$ the stabilizer of $t$ in $G$. For more details about pro-$p$ groups acting on pro-$p$ trees see \cite{Ribes2} and \cite{Melnikov}.

\medskip

We will need the following technical lemma, whose proof is similar to the proof of Lemma 2.7 in \cite{HZZ}. Recall that given a pro-$p$ group $G$, we denote by $d(G)$ the minimal number of topological generators of $G$.

\begin{lem}\label{decreasing}
Let $G$ be a finitely generated pro-$p$ group with $d(G)\geq 2$.
\begin{itemize}
\item [(a)] If $G=A\amalg_{C}B$ is a free amalgamated pro-$p$ product with $C$ procyclic, then $d(G)\geq d(A)+d(B)-1$.
\item[(b)] If $G=\textrm{HNN}(H, A,t)$ is a pro-$p$ HNN-extension with $A$ procyclic, then $d(G) \geq d(H)$.

\end{itemize}
\end{lem}
\begin{proof}
For a pro-$p$ group $H$ denote by $\bar{H}$ the Frattini quotient $H/{\Phi(H)}$.
\begin{itemize}
\item [(a)] Let $N$ be the kernel of the canonical homomorphism $\bar{A}\amalg \bar{B} \to \bar{G}$. Since $C$ is procyclic, the image $M$ of $N$ via the cartesian map $\bar{A}\amalg \bar{B} \to \bar{A} \times \bar{B}$ is also procyclic. The latter map induces an epimorphism from $\bar{G}$ to the elementary abelian pro-$p$ group $(\bar{A} \times \bar{B}) / M$. Hence
$d(G)=d(\bar{G})\geq d(\bar{A})+d(\bar{B})-1=d(A)+d(B)-1$.
\item [(b)] Suppose that $G=\textrm{HNN}(H, A,t)= \langle H, t ~ | ~ tat^{-1}=f(a) \rangle$, where $\langle a \rangle =A$.
Then there is an obvious epimorphism
$G\to (\bar{H}\times \bar{\langle t \rangle})/\langle \bar{t}\bar{a}{(\bar{t})}^{-1}{(\overline{f(a)})}^{-1}   \rangle$. Thus $d(G)\geq d(H)$.
\end{itemize}
\end{proof}

 Next we prove a preliminary result on the fundamental pro-$p$
group of a finite graph of finite $p$-groups. The fundamental
pro-$p$ group $\Pi_1(\G,\Gamma)$ of a finite graph of finite
$p$-groups $(\G, \Gamma)$ can be defined as the pro-$p$ completion
of the abstract (usual) fundamental group
$\Pi_1^{abs}(\G,\Gamma)$. Thus $G=\Pi_1(\mathcal{G}, \Gamma)$ has
the following presentation
\begin{displaymath}
\Pi_1(\mathcal{G}, \Gamma)=\langle \G(v), t_e\mid rel(\G(v)),
\partial_1(g)=\partial_0(g)^{t_e}, g\in \G(e), t_e=1 \  {\rm for}\ e\in
T\rangle;
\end{displaymath}
here $T$ is a maximal subtree of $\Gamma$ and $\partial_0:\G(e)\longrightarrow \G(d_0(e)),\partial_1:\G(e)\longrightarrow \G(d_1(e))$ are monomorphisms.

The fundamental group $\Pi_1(\G,\Gamma)$ acts on the standard pro-$p$ tree $S$ associated to it with vertex and edge stabilizers being conjugates of vertex and edge groups and such that $S/\Pi_1(\G,\Gamma)=\Gamma$ (see \cite{Melnikov}).

In contrast to the abstract case, the vertex groups of $(\G,
\Gamma)$ do not always embed in $\Pi_1(\G,\Gamma)$, i.e.,
$\Pi_1(\G,\Gamma)$ is not always proper. If
$\Pi_1^{abs}(\G,\Gamma)$ is residually $p$, then the vertex groups
of $(\G, \Gamma)$  embed in $\Pi_1(\G,\Gamma)$. Thus in the next
result we assume that $\Pi_1^{abs}(\G,\Gamma)$ is residually $p$.

\smallskip

\begin{lem}\label{graph}
Let $(\mathcal{G}, \Gamma)$ be a finite graph of finite $p$-groups
with cyclic edge groups $\G(e)$ such that $\G(e)\neq \G(v)$  for every edge $e$ in some maximal subtree $T_\Gamma$ of $\Gamma$ and every vertex $v$ incident to  $e$. Let
$G=\Pi_1(\mathcal{G}, \Gamma)$ be the fundamental pro-$p$ group
of $(\mathcal{G}, \Gamma)$. Then $d(G)$ tends to infinity whenever $|\Gamma|$ tends to infinity.\end{lem}
\begin{proof}
 Since the fundamental group $\Pi_1(\Gamma)$ is a free quotient group of $G$ of rank $|E(\Gamma)|-|V(\Gamma)|+1$,  if $|E(\Gamma)|-|V(\Gamma)|\to \infty$, then $d(G)\to \infty$ and we are done.
 Therefore we may assume that $|E(\Gamma)|-|V(\Gamma)|$ is bounded by some constant $k$. Since $G=\textrm{HNN}(\Pi_1(\mathcal{G}, T_\Gamma), \G(e), t_e, e\in \Gamma \backslash T_\Gamma)$ and $\G(e)$'s are cyclic, by Lemma \ref{decreasing} (b) it suffices to show that $d(\Pi_1(\mathcal{G}, T_\Gamma))$ grows.
  Thus we may assume that $\Gamma$ is a tree, i.e., $\Gamma=T_\Gamma$. Let $P$ be the set of
  pending vertices. Since $\G(e) \neq \G(v)$, we have that the free pro-$p$ product
 $\amalg_{l\in P}C_p$ of cyclic groups of order $p$ is a quotient of $\Pi_1(\mathcal{G}, T_\Gamma)$ (one can see this by factoring out the normal subgroup generated by $\G(e)$'s). Thus $|P|$ is bounded by $d(G)$ and so it suffices to prove the result for $T_\Gamma$ being a segment.
  Numerating its edges consequently, we note that the vertex groups of every odd edge generate non-abelian and so non-cyclic group $G_i$, $i=1,3,5\ldots$. Thus we have $\Pi_1(\mathcal{G}, T_\Gamma)=G_1\amalg_{\G(e_2)}G_3\amalg_{\G(e_4)}G_5\cdots$. Now the result follows by Lemma \ref{decreasing} (a).
\end{proof}

\begin{pro}\label{General}
Let $G$ be a finitely generated pro-$p$ group acting on a pro-$p$ tree $T$ with procyclic edge stabilizers.
 Then  $G$ is a surjective inverse limit $G=\varprojlim_{U}\Pi_1(\G_U,\Gamma)$ of fundamental groups of  finite graphs
of pro-$p$ groups
$(\G_U,\Gamma)$ (over the same finite graph $\Gamma$), where the connecting maps $\psi_{U,W}$ map each vertex group $\G_U(v)$
and each
edge group $\G_U(e)$ onto a conjugate of the vertex group $\G_W(v)$ and a conjugate of the edge group $\G_W(e)$ respectively.
Moreover, the maximal (by inclusion) vertex  stabilizers in $G$ are finitely generated and there are only finitely many of
them in $G$  up to conjugation. There are also finitely many edge stabilizers $G_e$, up to conjugation, whose images in
$\Pi_1(\G_U,\Gamma)$ are conjugates of edge groups and any other edge stabilizer is conjugate to a subgroup of one of these
$G_e$.

\end{pro}

\begin{proof}  For every open subgroup $U$ of $G$ consider
$\tilde U$, a subgroup generated by all intersections with vertex
stabilizers. Then by Proposition 3.5 and Corollary 3.6 in
\cite{Ribes2}, the quotient group $U/\tilde U$ acts freely on the
pro-$p$ tree $T/\tilde U$ and therefore it is free pro-$p$. Thus
$G_U:=G/\tilde U$ is virtually free pro-$p$. By Theorem 3.8 in
\cite{HZZ} it follows that $G_U$ is the fundamental pro-$p$
group $\Pi_1(\mathcal{G}_U, \Gamma_U)$ of a finite graph of finite
$p$-groups with cyclic edge stabilizers.  For a maximal subtree $T_{\Gamma_U}$ of $\Gamma_U$ we may assume that $\G_U(e)\neq \G_U(v)$  for every edge $e$ in $T_{\Gamma_U}$ and every vertex $v$ incident to  $e$ (if there is an edge $e\in T_\Gamma$ and a vertex $v$ incident to $e$ such that $\G_U(e)= \G_U(v)$, then we just collapse $e$).
Clearly  we have
$G=\varprojlim_{U}G_U$. Since $d(G_U)\leq d(G)$, by
Lemma \ref{graph} it follows that the number of vertices and edges
of $\Gamma_U$ is bounded for each $U$. Since there are only
finitely many finite graphs with bounded number of vertices and
edges, by passing to a cofinal system of $\{\Gamma_U \}$ if
necessary, we can assume that $\Gamma_U=\Gamma$ for each $U$. Fix a maximal subtree $T_\Gamma$ of $\Gamma$ and recall that $G_U=\Pi_1(\mathcal{G}_U, \Gamma)$ has the following presentation:
\begin{displaymath}
\Pi_1(\mathcal{G}_U, \Gamma)=\langle \G_U(v), t_U(e)\mid
rel(\G_U(v)),
\partial_1(g)=\partial_0(g)^{t_U(e)},
\end{displaymath}
\begin{displaymath}
g\in \G_U(e), t_U(e)=1 \  {\rm for}\ e\in
T_\Gamma \rangle.
\end{displaymath}

Now let $U$ and $W$ be open subgroups of $G$ such that $U\leq W$, let $v\in V(\Gamma)$ and let $\psi_{U,W}: G_U \to G_W$ be the natural epimorphism.
 Since $\mathcal{G}_U(v)$ is a finite $p$-group we have that $\psi_{U,W}(\mathcal{G}_U(v))$ also is a finite $p$-group, and so, by  Theorem (3.10) in \cite{Melnikov},
 it stabilizes a vertex (under the action of $G_W= \Pi_1(\mathcal{G}_W, \Gamma)$ on its associated pro-$p$ tree). Hence it is contained in a conjugate of some
 vertex group of $(\mathcal{G}_W, \Gamma)$. Since $\Gamma$ has only finitely many vertices, by passing to a cofinal system if necessary, for $U\leq W$ we have a
 homomorphism
$\mathcal{G}_{U}(v) \to {\mathcal{G}_{W}(v)}^{g_{U,W,v}}$, where
$g_{U,W,v}$ is some element of $\Pi_1(\mathcal{G}_W, \Gamma)$.

\smallskip
Let $e\in E(\Gamma)$ and suppose that $d_0(e)=u$ and $d_1(e)=v$.
Then, since $\mathcal{G}_{U}(e)=\mathcal{G}_{U}(u)\cap
\mathcal{G}_{U}(v)$, for $U\leq W$ we have
\begin{displaymath}
\psi_{U,W}(\mathcal{G}_{U}(e))\leq {\mathcal{G}_{W}(u)}^{g_{U,W,u}} \cap {\mathcal{G}_{W}(v)}^{g_{U,W,v}}\eqno{(1)}
\end{displaymath}

 Thus, as in the case with vertex groups, for $U\leq W$ (if necessary we pass to a cofinal system), the group $\mathcal{G}_{U}(e)$ maps to the group
 $\mathcal{G}_{W}(e)$, up to conjugation. Thus for every $e$ we have an inverse system $\{\G_U(e)^{g_U}\mid g_U\in G_U\}$ of conjugates of $\G_U(e)$. The inverse limit of
 these families, for every $e\in E(\Gamma)$, gives the family $\{\G_e\}$ of groups closed under the conjugation by elements of $G$.  Let us choose a representative
 $G(e)$ of  $\{\G_e\}$. Its images on $\Pi_1(\mathcal{G}_U, \Gamma)$ under the projection maps form the inverse system $\{\mathcal{G}_{U}'(e)\}$ (for each $e \in E(\Gamma))$; this inverse system is surjective by Lemma 2.1 (a)  in \cite{HZZ}, if $G(e)\neq 1$.
   For each $U$, the group $\mathcal{G}_{U}(e)$  is the stabilizer of an edge of the pro-$p$ tree $T/\tilde U$ by Theorem 3.8 in \cite{HZZ} and therefore so is $\mathcal{G}_{U}'(e)$. Hence $G(e)$
   stabilizes an edge of the pro-$p$ tree
   $T=\varprojlim_{U}T/\tilde U$. If $G(e)=1$, then we can factor out the normal closure of  $\G_U(e)$, since by Lemma 2.1 in \cite{HZZ} we have  $G=\varprojlim_{U} G_U/(\G_U(e))^{G_U}$ for such $e$. Thus we may assume that $\{\G'_U(e)\}$ is surjective for every $e$. It follows  that $G(e)$ is the stabilizer in $G$ of an edge
   of $T$.

 \smallskip

 Note that the homomorphism  $\mathcal{G}_{U}(v) \to {\mathcal{G}_{W}(v)}^{g_{U,W,v}}$ is an epimorphism. Indeed, suppose that this homomorphism is not surjective.
  Then, since ${\mathcal{G}_{W}(v)}^{g_{U,W,v}}$ is a finite $p$-groups,  $\psi_{U,W}(\mathcal{G}_{U}(v))$ is contained in a maximal subgroup of
  ${\mathcal{G}_{W}(v)}^{g_{U,W,v}}$, which is normal and of index $p$. Using the fact that the homomorphism
  $\mathcal{G}_{U}(e) \to {\mathcal{G}_{W}(e)}^{h_{U,W,e}}$ is an
  epimorphism, by factoring out the normal closure of all vertex groups
  of $\Pi_1(\mathcal{G}_W, \Gamma)$ except $\G_W(v)$,
 it is easy to see that we have a contradiction, since $\psi_{U,V}$ is an
 epimorphism.

\smallskip
 For every vertex $v$ we have an inverse system $\{\G_U(v)^{g_u}\mid g_u\in G_U\}$ of conjugates of $\G_U(v)$. The inverse
limit of
 these families gives the family $\{\G_v\}$ of groups closed under the conjugation by elements of $G$.  Let us choose a
representative $G(v)$ of  $\{\G_v\}$. Its images on $\Pi_1(\mathcal{G}_U, \Gamma)$ under the projection maps form the
surjective inverse system $\{\mathcal{G}_{U}'(v)\}$.
 For each $U$, the group $\mathcal{G}_{U}(v)$  is the stabilizer of a vertex of the pro-$p$ tree $T/\tilde U$ by Theorem 3.8
 in \cite{HZZ} and therefore $\mathcal{G}_{U}'(v)$ as a conjugate of $\G_U(v)$ is the stabilizer of a vertex of
$T/\tilde U$. Hence $G(v)$ is the stabilizer in $G$ of a vertex of $T=\varprojlim_{U}T/\tilde U$.

\smallskip

 Finally, note that from the fact that $d(G_U)\leq d(G)$ for each $U$ and Lemma \ref{decreasing} it follows easily that
 $G(v)$ is finitely generated for each $v\in V(\Gamma)$. To prove the last statement of the theorem, let $H$ be the stabilizer of a vertex $w$ in $T$. Then $\psi_U(H)$ is the stabilizer of the image of $w$ in $T/\tilde U$ and in particular it is finite.
 Therefore by Theorem 3.10 in \cite{Melnikov} it is conjugate to a subgroup of a vertex group $\G_U(v)$ and so to a
subgroup of  $\G'_U(v)$. Therefore $H$ is conjugate to a subgroup of $G(v)$. If $H$ is the stabilizer of an edge of $T$ one uses a similar argument combined with equation (1).
    This finishes the proof of the proposition.

\end{proof}

We now introduce two separate subsections to be treated  separately: the case of acylindrical action (that will be used in the rest
of the
paper)
and the case when $G$ is generated by its vertex stabilizers.

\subsection{Acylindrical action}

\begin{definition} Let $G$ be a pro-$p$ group acting on a pro-$p$ tree $T$. We say that this action is $n$-acylindrical if
for every non-trivial edge stabilizer $G_e$ the subtree of fixed points $T^{G_e}$ (cf. Theorem 3.7 in \cite{Ribes2})
 has diameter $n$. Note that
by Corollary 4 in \cite{HZ} this means that any element $1\neq
g\in G$ can fix at most $n$ edges  in any (profinite) geodesic
$[v,w]$ of $S(G)$.
\end{definition}

\begin{lem}\label{acylindric} Let $n$ be a natural number and $G$ be a finitely generated pro-$p$ group acting $n$-acylindrically on a pro-$p$ tree $T$
with procyclic edge stabilizers such that $T/G$ has finite
diameter.
 Then $G$ is the fundamental pro-$p$ group of a finite graph of
pro-$p$ groups $(\G, \Delta)$ with procyclic edge groups and
finitely generated vertex groups. Moreover, the vertex and edge
 groups of $(\G, \Delta)$ are  stabilizers of certain vertices and edges of $T$ respectively.\end{lem}

 \begin{proof}
 Note first that $T/G$ is connected as an abstract graph (see Corollary 4 in \cite{HZ}) and
 therefore every finite cover of it is also connected. It follows
 that $\pi_1(T/G)$ is just the pro-$p$ completion of the ordinary
 fundamental group $\pi_1^{abs}(T/G)$ (see Proposition 2.1 in \cite{Z}). By Lemma \ref{General} there are finitely many maximal stabilizers of vertices
 $G_{w_1}, G_{w_2}, \ldots G_{w_m}$ up to conjugation. 
   Let $C_1,C_2,\ldots C_n$ be
 simple circuits that are free generators of $\pi_1^{abs}(T/G)$
 and let $v_1,v_2,\ldots v_m$ be the images of $w_1,\ldots, w_m$
 in $T/G$. Put $\Delta$ to be a minimal connected subgraph of $T/G$ containing $C_1,C_2,\ldots C_n$
 and  $v_1,v_2,\ldots v_m$; clearly $\Delta$ is finite. By the pro-$p$
version of Lemma 2.14 in \cite{CZ} for any connected component
$\Omega$ of the preimage of $\Delta$ in $T$ and its setwise
stabilizer $Stab_G(\Omega)$ we have
$\Omega/Stab_G(\Omega)=\Delta$. By  Proposition 4.4 in
\cite{ZM:90} a pro-$p$ group acting on a pro-$p$ tree cofinitely
is the fundamental group of a finite graph of groups in a standard
manner, i.e., in our case $Stab_G(\Omega)=\Pi_1(\G,\Delta)$. More
precisely,
   $\Delta$ admits a connected transversal $D$ in $\Omega$ with
$d_0(e)\in D$ for every $e\in D$. This
gives the standard structure of a graph of pro-$p$ groups
$(\G,\Delta)$ on $\Delta$, where the vertex and edge groups are
stabilizers of vertices  and edges of $D$  and we have
$$\Pi_1(\G,\Delta)=\langle G_v, x_e\in Stab_G(\Omega)\mid v\in V(D), x_ed_1(e)\in D,$$
$$\ {\rm for}\ e\in E(D)
\ {\rm with}\ d_1(e)\not\in D\rangle. $$ 
Let $u_1,\ldots u_m$  be the preimages of $v_1\ldots v_m$ in
$D$. Then  $G_{u_1}, G_{u_2}, \ldots G_{u_m}$ are conjugates of $G_{w_1}, G_{w_2}, \ldots G_{w_m}$, so that every vertex stabilizer of $G$ up to conjugation is contained in one of them.
Therefore  $G$ is generated by $\pi_1^{abs}(T/G)$ and $G_{u_1}, G_{u_2}, \ldots G_{u_m}$ (see it modulo Frattini). Thus we have $$G=\langle G_{u_i}, x_e\in Stab_G(\Omega)\mid
i=1,\ldots m, x_ed_1(e)\in D,$$
$$\ {\rm for}\ e\in E(D)
\ {\rm with}\ d_1(e)\not\in D\rangle$$ and so
$G=\Pi_1(\G,\Delta)$.
 \end{proof}

\begin{thm}\label{StructureTheorem-General}
Let $n$ be a natural number and $G$ be a finitely generated pro-$p$ group acting $n$-acylindrically on a pro-$p$ tree $T$
with procyclic edge stabilizers.
 Then $G$ is the fundamental pro-$p$ group of a finite graph of
pro-$p$ groups $(\G, \Gamma)$ with procyclic edge groups and finitely generated vertex groups. Moreover, the vertex and edge
 groups of $(\G, \Gamma)$ are  stabilizers of certain vertices and edges of $T$ respectively, and stabilizers of vertices
 and edges of $T$ in $G$ are conjugate to subgroups of vertex and edge groups of $(\G, \Gamma)$ respectively.

\end{thm}

\begin{proof}
 \medskip
By Proposition \ref{General} there are only finitely many maximal
by inclusion edge and  vertex stabilizers in $G$ up to
conjugation. Then, since the action is  $n$-acylindrical,
$T^{G_e}$ has diameter at most  $n$  for every non-trivial edge
stabilizer $G_e$. It follows that $\bigcup_{G_e\neq 1}
T^{G_e}/G$ has finite diameter. Indeed, since there are only
finitely many maximal edge stabilizers up to conjugation, it
suffices to show that for a maximal edge stabilizer
$G_{me'}$ stabilizing an edge $e'$, the tree $\bigcup_{G_e\leq
G_{me'}} T^{G_e}$ has finite diameter. But for $G_e\leq G_{me'}$
the geodesic $[e,e']$ is stabilized by $G_e$ (cf. Corollary 3.8 in
\cite{Ribes2}) and so has length not more than $n$.

 Thus $\bigcup_{G_e\neq 1}
T^{G_e}/G$ has finite diameter and finitely many connected components. It follows that the closure $\Delta$ of it has also finite diameter (see appendix in \cite{HZ}) and finitely many connected components.

 Let $\Delta_\alpha$ be a connected component of $\Delta$. By the pro-$p$
version of Lemma 2.14 in \cite{CZ} for any connected component
$\Omega_\alpha$ of the preimage of $\Delta_\alpha$ in $T$ and its
setwise stabilizer $Stab_G(\Omega_\alpha)$ we have
$\Omega_\alpha/Stab_G(\Omega_\alpha)=\Delta_\alpha$.   Collapsing
all connected components of the preimage of $\Delta$ in $T$, by
the Proposition on page 486 in  \cite{Zalesskii:89} we get a
pro-$p$ tree $\bar T$ on which $G$ acts with trivial edge
stabilizers (since $\bar T^{G_e}$ is connected for every $e\in
E(\bar T)$ by Theorem 3.7 in \cite{Ribes2}), so by Proposition
2.12 in \cite{HZZ} we have that  $G$ is a free pro-$p$ product
$$G=(\coprod_{\alpha} Stab_G(\Omega_\alpha))\amalg (\coprod_{v\not\in \bigcup_\alpha D_\alpha} G(v))\amalg \pi_1(\bar T/G).$$

 Therefore $Stab_G(\Omega_\alpha), \pi_1(\bar T/G)$ and  $G(v)$ for $v\not\in \bigcup_\alpha D_\alpha$ are finitely generated.

 By Lemma \ref{acylindric} we have that
 $Stab_G(\Omega_\alpha)=\Pi_1(\G,\Delta_\alpha)$ is the fundamental group of a finite graph
of groups in a standard manner, where the vertex and edge groups
are stabilizers of vertices  and edges of $D_\alpha$  and  so
$$\Pi_1(\G,\Delta_\alpha)=\langle G_v, x_e\in Stab_G(\Omega_\alpha)\mid v\in V(D_\alpha), x_ed_1(e)\in D_\alpha,$$
$$\ {\rm for}\ e\in E(D_\alpha)
\ {\rm with}\ d_1(e)\not\in D_\alpha\rangle.$$

Since the free pro-$p$ product of the fundamental pro-$p$ groups of finitely many finite graphs of pro-$p$ groups is again the fundamental pro-$p$ group of a finite graph of pro-$p$ groups, we have the needed structure of the fundamental pro-$p$ group of a finite graph of pro-$p$ groups on $G$ in this case.

The last part of the theorem follows from Proposition \ref{General}.
\end{proof}

\subsection{Generation by stabilizers}\hskip10cm

\medskip
If $G$ is generated by vertex stabilizers we can prove the structure theorem without $n$-acylindricity. To accomplish this we need first the following.

\begin{lem}\label{presentation}
Let $(\mathcal{G}, \Gamma)$ be a finite tree of finite $p$-groups
 and let
$G=\Pi_1(\mathcal{G}, \Gamma)$ be the fundamental pro-$p$ group of
$(\mathcal{G}, \Gamma)$. Let  $G(\Gamma)=\coprod_{v\in V(\Gamma)} \G(v)$ be a free pro-$p$ product
and let $\psi:G(\Gamma)\longrightarrow G$ be the
epimorphism sending $\G(v)$  to their copies in $G$.  Suppose there is a collection
$\{G(v)=\G(v)^{g_v}, v\in V(\Gamma), g_v\in G(\Gamma)\}$ of conjugates of free factors of $G(\Gamma)$
  and a collection $\{G(e)=\G(e)^{g_e}, e\in E(\Gamma), g_e\in G\}$ of conjugates of edge groups of $G$
 such that
$\psi(G(d_1(e)))\cap
\psi(G(d_0(e)))=G(e)$. Then the kernel
of $\psi$ is generated by the set of elements
$\psi_{1,e}^{-1}(g^{-1})\psi_{0,e}^{-1}(g)$,
where $g\in G(e)$ and $\psi_{i,e}=\psi_{|G(d_i(e))}$, $i=0,1$.
\end{lem}

\begin{proof} Note that
$\psi(\psi_{1,e}^{-1}(g^{-1})\psi_{0,e}^{-1}(g))=g^{-1}g=1$ and
so the elements $\psi_{1,e}^{-1}(g^{-1})\psi_{0,e}^{-1}(g)$
belong to the kernel of $\psi$. This means that $\psi$ factors via the natural quotient homomorphism
$\pi:G(\Gamma)\longrightarrow \Pi$ modulo the normal closure of the elements
$\psi_{1,e}^{-1}(g^{-1})\psi_{0,e}^{-1}(g)$, i.e. there exists
 a natural epimorphism $f:\Pi\longrightarrow G$ such that $f\pi=\psi$.

Define now a tree of pro-$p$ groups $(\G',\Gamma)$  as follows. Put
$\G'(v)=\pi(G(v))$, $\G'(e)=\pi(\psi_{0,e}^{-1}(G(e))$  and define
$\partial_0,\partial_1$ to be the natural embeddings of $\G'(e)$
into $\G'(d_0(e))$ and into $\G'(d_1(e))$. Then  the relations
$$\psi_{1,e}^{-1}(g^{-1})\psi_{0,e}^{-1}(g),$$
where $g\in \psi(G(e))$,
define on  $\Pi$ the structure of  the fundamental group $\Pi_1(\G',\Gamma)$ of the graph $(\G',\Gamma)$ of groups.

Let $F$ be an open free pro-$p$ subgroup of $G$. Then $f^{-1}(F)$ is an open free pro-$p$ subgroup of $\Pi$ of the same
index as the index of $F$ in $G$. Then by the Euler characteristic formula (cf. Exercise 3 on page 123 in \cite{Serre-77}),
that holds  here since our groups are the pro-$p$ completions of the corresponding abstract groups, we have
  $$rank(F)-1=|G:F|(\sum_{e\in E(\Gamma)} 1/|\G(e)|- \sum_{v\in V(\Gamma)} 1/|\G(v)|)=$$
  $$|\Pi:f^{-1}(F)| (\sum_{e\in E(\Gamma)}1/|G(e)| - \sum_{v\in V(\Gamma)} 1/|G(v)|)=rank(f^{-1}(F))-1.$$ Thus
the free pro-$p$ groups $F$ and $f^{-1}(F)$ have the same rank and therefore they are isomorphic. Since the kernel of $f$
is torsion free, $f$ is an isomorphism, as desired.
\end{proof}

\begin{thm} Let $G$ be a finitely generated pro-$p$ group acting on a pro-$p$ tree $T$ with procyclic edge stabilizers.
 Suppose $G$ is generated by its vertex stabilizers. Then $G$ is the fundamental pro-$p$ group of a finite tree of
pro-$p$ groups $(\G, \Gamma)$ with procyclic edge groups and finitely generated vertex groups.
Moreover, the vertex and edge groups of $(\G, \Gamma)$ are  stabilizers of certain vertices and edges of $T$ respectively, and stabilizers of vertices and edges of $T$ in $G$ are conjugate to subgroups of vertex and edge groups of $(\G, \Gamma)$ respectively.
\end{thm}
\begin{proof} By Proposition \ref{General} the group $G$  is a surjective inverse limit $G= \varprojlim_{U}G_U$, where
$G_U=\Pi_1(\G_U,\Gamma)$ is the  fundamental group of a finite graph
of pro-$p$ groups
$(\G_U,\Gamma)$, where the connecting maps $\psi_{U,W}$ map each vertex group $\G_U(v)$
and each
edge group $\G_U(e)$ onto a conjugate of the vertex group $\G_W(v)$ and a conjugate of the edge group $\G_W(e)$ respectively.
Moreover,
there are only finitely many maximal by inclusion vertex stabilizers
in $G$ up to conjugation and also finitely many up to conjugation edge stabilizers $G_e$ whose images in
$\Pi_1(\G_U,\Gamma)$ are conjugates of edge groups and any other edge stabilizer is conjugate to a subgroup of one of these
$G_e$. Keeping the notation of the proof of  Proposition \ref{General} we denote  by
$G(e)$, $G(v)$  some
representatives of them. Note  that in this case, by Proposition 3.5 in \cite{Ribes2}, it follows that $\Gamma$ is a tree.

\medskip
\emph{Claim } We can choose the representatives $G(e)$ and
$G(v)$ such that for $e\in E(\Gamma)$ one has
 $G(e)=G(d_0(e))\cap G(d_1(e))$.

\smallskip
\emph{Proof of the claim}.

Let $D$ be a maximal subtree of $\Gamma$ such that this holds for all
$e\in E(D)$. We show that $D=\Gamma$. Suppose not. Then there exists
$e\in E(\Gamma)\setminus E(D)$ such that a vertex $v$ of $e$ is in $D$.
 Let $\G'_U(e)=\G_U(e)^{h_U}$ be the image
of $G(e)$ in $G_U$. Then clearly $\G_U(v)^{h_U}$ contains
$\G_U(e)^{h_U}$. Since $\G'_U(v)$ is a conjugate of $\G_U(v)$, it
follows that the set $X_U$ of elements $x_U\in G_U$ such
that $\G'_U(e)^{x_U}\leq \G'_U(v)$ is non-empty and clearly these sets form an
inverse system $\{ X_U \}$. It follows that the inverse limit $X$  of $\{ X_U \}$ is
non-empty and $G(e)^x\leq G(v)$ for any $x\in X$. So we replace
$G(e)$ by $G(e)^x$ (in this way $\G'_U(e)$ is replaced by $\G'_U(e)^{x'_U}$, where $x'_U$ is the image of $x$ in $G_U$).
Let $w$ be the other vertex of $e$. Similarly, there is an inverse system $\{Y_U \}$
of non-empty subsets of $G_U$ such that $\G'_U(e)\leq
{\G'_U(w)}^{y_U}$ for each $y_U \in Y_U$. Then the
inverse limit $Y$ of $\{Y_U \}$ is non-empty and for each $y\in Y$ we have $G(v)\cap G(w)^{y}=G(e)$
(since $\G'_U(v)\cap \G'_U(w)^{y_U}=\G_U(e)$ for every
$U$). Then $D\cup \{e\}\cup \{w\}$ satisfies the statement,
contradicting the maximality of $D$.

\medskip

Now consider the projection $\psi_U: G \to G_U$, and let
$\G^*_{U}(v)=\psi_U(G(v))$,
$\G^*_{U}(e)=\psi_U(G(e))$ for $G(v), G(e)$ being as in the Claim.  Let
\begin{displaymath}
G_U(\Gamma):=\coprod_{v\in V(\Gamma)}\G_{U}(v)
\end{displaymath}
and let $f_U: G_U(\Gamma) \to \Pi_1(\mathcal{G}_{U}, \Gamma)$
 be the  homomorphism defined by sending $\G_U(v)$ to their copies
 in $\Pi_1(\mathcal{G}_{U}, \Gamma)$. We choose an element $g_{U,v}\in G_U(\Gamma)$ such that
 $f_U(\G_U(v)^{g_{U,v}})=\G^*_U(v)$. Put $G_U(v)=\G_U(v)^{g_{U,v}}$. Since free products in the pro-$p$ case do not depend
on the conjugation of the factors
 (see Exercise 9.1.22 in \cite{Ribes1}), we have
%\begin{displaymath}
 $G_U(\Gamma)=\coprod_{v\in V(\Gamma)}G_{U}(v)$.
%\end{displaymath}

 Now let $U$ and $W$ be open subgroups of $G$ such that
$U\leq W$. Then the maps $G_{U}(v) \to
G_{W}(v)$  induce
an epimorphism
 $\varphi_{U,W}:G_U(\Gamma)\to G_{W}(\Gamma)$,
which gives the following commutative diagram
\begin{displaymath}
\xymatrix {G_{U}(\Gamma) \ar[d]^{f_U} \ar[r]^{\varphi_{U,W}}\ar[d]^{f_U}&
G_{W}(\Gamma)\ar[d]^{f_W} \\
          \Pi_1(\mathcal{G}_{U}, \Gamma) \ar[r]^{\psi_{U,W}} & \Pi_1(\mathcal{G}_{W}, \Gamma)}
\end{displaymath}
 Let $G(\Gamma):=\coprod_{v\in
V(\Gamma)}G(v)$. Then the maps $G(v) \to {\mathcal{G}_{U}(v)}^{g_{U,v}}$  induce an epimorphism
 $\varphi_U:G(\Gamma) \to G_{U}(\Gamma)$
such that $\varphi_{U,W}\varphi_U=\varphi_W$. Thus we have a
surjective inverse system
$\{ G_{U}(\Gamma)\},$
 which
by Lemma 9.1.5 in \cite{Ribes1} has inverse limit
\begin{displaymath}
G(\Gamma)=\varprojlim_{U} G_{U}(\Gamma)) =\coprod_{v\in V(\Gamma)}G(v).
\end{displaymath}

 Note that $\mathcal{G}^{*}_{U}(e)$ and
$\mathcal{G}^{*}_{U}(v)$ are conjugates in $G_U$ of
$\mathcal{G}_{U}(e)$ and $\mathcal{G}_{U}(v)$ respectively and by the Claim the
relations of Lemma \ref{presentation} hold for
$\mathcal{G}^{*}_{U}(e)$ and $G_U(v)$.
It follows that $f_U$ is the epimorphism defined by just  imposing on $G_U(\Gamma)$
the amalgamation relations
 $f_{U,1,e}^{-1}(g)=f_{U,0,e}^{-1}(g)$ for $g\in \G^{*}_U(e)$, $e\in E(\Gamma)$, where
$f_{U,i,e}=(f_U)_{|G_U(d_i(e))}$, $i=0,1$;
this means that the kernel of $f_U$ is generated by the relators
$f_{U,1,e}^{-1}(g^{-1})f_{U,0,e}^{-1}(g)$ for
$g\in \G^{*}_U(e)$, $e\in E(\Gamma)$. Let $f:G(\Gamma)\longrightarrow G$ be the
 epimorphism given as the projective limit of the epimorphisms $ f_U $. Put $f_{i,e}=f_{|G(d_i(e))}$, $i=0,1$.
It follows that imposing on $G(\Gamma)$ the
relations
 $f_{1,e}^{-1}(g^{-1})f_{0,e}^{-1}(g)=1$, where $g\in G(e)$ defines exactly $f$. This gives the
desired structure  (i.e., presentation)  of the
 fundamental group of a graph of groups on $G=\Pi_1(\G,\Gamma)$, with vertex end edge groups $G(v)$ and $G(e)$ and with the
corresponding natural embeddings.

The rest of the proof follows directly from Proposition \ref{General}.
 \end{proof}

\section{The decomposition theorem for pro-$p$ groups from the class $\mathcal{L}$}

In this section we prove Theorem B and parts $(1)$, $(2)$,
$(3)$ and $(4)$ of Theorem C, stated in the introduction.

We say that the amalgamated free pro-$p$ product $A\amalg_{C}B$ is proper if $A$ and $B$ embed in $A\amalg_{C}B$.
 Ribes proved that an amalgamated free pro-$p$ product with procyclic amalgamation is proper
(see Theorem 3.2 in \cite{Ribes3}).

It is worth to recall the definition of the class $\mathcal{L}$ of pro-$p$ groups \cite{Kochloukova2}. Denote by
$\mathcal{G}_0$ the class of all free pro-$p$ groups of finite
rank. We define inductively the class $\mathcal{G}_n$ of pro-$p$
groups $G_n$ in the following way: $G_n$ is a free pro-$p$
amalgamated product $G_{n-1}\amalg_{C}A$, where $G_{n-1}$ is any
group from the class $\mathcal{G}_{n-1}$, $C$ is any
self-centralized procyclic pro-$p$ subgroup of $G_{n-1}$ and $A$
is any finite rank free abelian pro-$p$ group such that $C$ is a
direct summand of $A$. The class of pro-$p$ groups $\mathcal{L}$
consists of all finitely generated pro-$p$ subgroups $H$ of some
$G_n\in \mathcal{G}_n$, where $n\geq 0$. If $n$ is  minimal
with the property  that $H\leq G_n$ for some $G_n\in \G_n$, we say that $H$ has
weight $n$.
Then $H$ is a  subgroup of a free
amalgamated pro-$p$ product $G_n=G_{n-1}\amalg_{C}A$, where
$G_{n-1}\in \mathcal{G}_{n-1}$, $C\cong \mathbb{Z}_p$ and
$A=C\times B\cong \mathbb{Z}_p^m$. As was mentioned above, by Theorem 3.2 in
\cite{Ribes3}, this amalgamated pro-$p$ product is proper. Thus $H$
acts naturally on the pro-$p$ tree $T$ associated to $G_n$ (see
\cite{Ribes2}) and its edge stabilizers are procyclic.

\begin{lem}\label{2-acylindrical} Let $H$ and $T$ be as above. Then the action of $H$ on $T$ is 2-acylindrical.
 \end{lem}

\begin{proof} It suffices to prove that the action of $G_n$ on $T$ is 2-acylindrical.
Let $G_e$ be a non-trivial edge stabilizer. If the diameter of $T^{G_e}$ is bigger than 2, then it contains a non-pending
vertex $v$ whose
stabilizer is conjugate to $G_{n-1}$ and so we may assume without loss of generality that it is $G_{n-1}$. Let $e'$ be another edge incident
to $v$ stabilized by $G_e$. Then $ge=e'$ for some $g\in G_{n-1}$ and so $g\in N_{G_{n-1}}(G_e)$ (we use here that $G_e$ is
procyclic). By Theorem 5.1 in \cite{Kochloukova2} it follows that $N_{G_{n-1}}(G_e)=C_{G_{n-1}}(G_e)=G_e$. Thus $e=e'$, a contradiction.
 \end{proof}

\begin{thm}\label{StructureTheorem}
Let $G$ be a pro-$p$ group from the class $\mathcal{L}$. If $G$ has weight $n\geq 1$, then it is the fundamental pro-$p$ group of a finite graph of
pro-$p$ groups that has infinite procyclic or trivial edge groups and finitely generated vertex groups. Moreover, if $G$ is non-abelian, then it has at least one vertex group that is a non-abelian pro-$p$ group and all the non-abelian vertex groups of $G$ are pro-$p$ groups from the class $\mathcal{L}$ of weight $\leq n-1$.
\end{thm}
\begin{proof} By Lemma \ref{2-acylindrical}, the action of $G$ on the standard pro-$p$ tree $T$ associated with $G_n$ is
2-acylindrical; so by Theorem \ref{StructureTheorem-General} it follows that  $G=\Pi_1(\G,\Gamma)$ is the fundamental pro-$p$ group of a finite graph of pro-$p$ groups with procyclic edge groups and finitely generated vertex groups. Moreover, each vertex group of $G$ is a vertex stabilizer of $G$ in $T$; thus it is a pro-$p$ group from the class $\mathcal{L}$ contained in a subgroup of $G_n=G_{n-1}\amalg_{C}A$ conjugate to $G_{n-1}$ or $A$. If it is non-abelian, then it must be contained in a subgroup of $G_n$ conjugate to $G_{n-1}$ and so it has weight $\leq n-1$. Thus, in order to finish the proof it remains to show that at least one of the vertex groups of $G$ is non-abelian.

Let $T_\Gamma$ be a maximal subtree of $\Gamma$. By
collapsing the fictitious edges of $T_\Gamma$ (i.e., edges whose edge group
is equal to the vertex group of a vertex of this edge) we may assume
that all vertex groups contain properly edge groups for
incident edges. Then if all vertex groups are abelian we can have
at most one vertex in $\Gamma$ because otherwise the centralizer of the edge group $\G(e)$ is not abelian for any edge $e\in T_\Gamma$, contradicting Theorem 5.1 in
\cite{Kochloukova2}. Thus we may assume that $T_\Gamma$ has only one vertex.  Let $H$ be the vertex group of this unique
vertex and $A$ the edge group (which is procyclic). Then $G=\textrm{HNN}(H, A, t)$.  If $H$ is not procyclic, then
since $\langle H, H^t \rangle=H\amalg_{A} H^t$ (cf. Proposition 4.4 in \cite{ZM:90}), we get once more a contradiction by
Theorem 5.1 in
\cite{Kochloukova2}. Now suppose that $H$ is procyclic. Then we must have $A=A^t$. Hence $A$ is normalized by $t$ and
therefore it is central in $G$. By Theorem 5.1
 in \cite{Kochloukova2} it follows that $G$ is abelian, which is a contradiction. Thus at least one of the vertex groups of
 $G$ is non-abelian.
\end{proof}

Now let $G$ be as in the above theorem. Using the theorem and induction we can deduce that there are only finitely many conjugacy classes of non-procyclic maximal abelian subgroups of $G$.  Indeed, let us suppose that this result holds for all pro-$p$ groups from the class $\mathcal{L}$ of weight $\leq n-1$ (and $G$ has weight $n$). Let $H$ be a non-procyclic maximal abelian subgroup of $G$ and consider the action of $G$ on $T$ according to the first paragraph in the above proof. Then $H$ is a subgroup of $G_{n-1}\amalg_{C}A$ and so, by Corollary 5.5 in \cite{Kochloukova2}, the group $H$ is conjugate in $G_n$ to a subgroup of $G_{n-1}$ or to a subgroup of $A$. Thus it stabilizes a vertex of $T$, and therefore, by Theorem \ref{StructureTheorem-General}, it is contained in a conjugate of a vertex group $\G(v)$ of $(\G,\Gamma)$. The non-abelian vertex groups of $G$ have weight $\leq n-1$ and therefore, by the induction hypothesis, they have only finitely many conjugacy classes of non-procyclic maximal abelian subgroups. Since $G$ has only finitely many vertex groups, the result follows. Let us record this result in the following.

\begin{cor}
Let $G$ be a pro-$p$ group from the class $\mathcal{L}$. Then there are only finitely many conjugacy classes of non-procyclic maximal abelian subgroups of $G$.
\end{cor}

Recall that if $cd(G)<\infty$ and if $\textrm{dim} _{\mathbb{F}_p} H^k(G,\mathbb{F}_p)< \infty$ for all $k\geq 0$, then the \emph{Euler-Poincar\'e characteristic} of $G$ is defined by
\begin{displaymath}
\chi(G):=\sum_{k=0}^{\infty} (-1)^k \textrm{dim} _{\mathbb{F}_p} H^k(G,\mathbb{F}_p).
\end{displaymath}
Moreover, if $G$ is the fundamental pro-$p$ group of a finite graph of pro-$p$ groups $(\G, \Gamma)$ such that the Euler-Poincar\'e characteristic is well defined for the vertex and edge groups, then the action of $G$ on the standard tree $S(G)$ implies the formula
\begin{displaymath}
\chi(G)=(\sum_{v\in V(\Gamma)}\chi(\G(v)))-(\sum_{e\in E(\Gamma)}\chi(\G(e))).
\end{displaymath}

The first part of the following theorem coincides with Theorem 8.1 in \cite{Kochloukova2}, while the second part generalizes Theorem 8.2 and gives an affirmative answer to the question 9.3 of the same paper.
\begin{thm}\label{Euler}
Let $G$ be a pro-$p$ group from the class $\mathcal{L}$. Then $G$ has a non-positive Euler-Poincar\'e characteristic. Moreover $\chi(G)=0$ if and only if $G$ is abelian.
\end{thm}
\begin{proof}
Clearly $\chi(G)=0$ if $G$ is abelian. Thus it suffices to show
that $\chi(G)<0$ whenever $G$ is non-abelian. We will prove this
using induction on the weight $n$ of the group $G$. Suppose that
$G$ is non-abelian. If $n=0$, then $G$ is a non-abelian free
pro-$p$ group and so we have $\chi(G)=1-d(G)<0$. Now let $n\geq 1$
and suppose that every non-abelian pro-$p$ group from the class
$\mathcal{L}$ which has weight less than $n$ has a negative Euler-Poincar\'e
characteristic. By Theorem \ref{StructureTheorem}, the group $G$
is the fundamental pro-$p$ group of a finite graph of pro-$p$ groups $(\G, \Gamma)$ with
infinite procyclic or trivial edge groups and whose vertex groups are either finitely generated free abelian pro-$p$ groups or non-abelian
pro-$p$ groups from the class $\mathcal{L}$ of weight $\leq n-1$. Moreover, there is at least one non-abelian vertex group, say $\G(v)$. Thus, by the induction hypothesis, we have $\chi(\G(v))<0$. Now by the Euler-Poincar\'e characteristic formula we have
\begin{displaymath}
\chi(G)=(\sum_{x\in V(\Gamma)}\chi(\G(x)))-(\sum_{e\in E(\Gamma)}\chi(\G(e)))=(\sum_{x\in V(\Gamma)}\chi(\G(x)))-(\sum_{e\in E(\Gamma)} 0) \leq \chi(\G(v))<0.
\end{displaymath}
\end{proof}

Let $r(G)$ denote the minimal number of relations of $G$, i.e,
\begin{displaymath}
r(G):=\rm{inf} \{ |R| ~ | ~ G \textrm{ has a presentation } \langle X ~ | ~ R \rangle \textrm{ with } |X|=d(G) \}.
\end{displaymath}
It is a well known fact that $d(G)=\textrm{dim}_{\mathbb{F}_p}H^1(G, \mathbb{F}_p)$, and if $G$ is finitely generated, then $r(G)=\textrm{dim}_{\mathbb{F}_p}H^2(G, \mathbb{F}_p)$ (see \cite{Serre1}). Recall that if $G$ is a finitely presented pro-$p$ group, then the \emph{deficiency} of $G$ is defined by
\begin{displaymath}
\textrm{def}(G):=d(G)-r(G)=\textrm{dim}_{\mathbb{F}_p}H^1(G, \mathbb{F}_p)-\textrm{dim}_{\mathbb{F}_p}H^2(G, \mathbb{F}_p).
\end{displaymath}

\begin{lem}\label{def}
Let $G$ be a finitely generated pro-$p$ group with $d(G)\geq 2$.
\begin{itemize}
\item [(a)] If $G=A\amalg_{C}B$ where $C$ is procyclic, then $\textrm{def}(G)\geq \textrm{def}(A)+\textrm{def}(B)-2$.
\item[(b)] If $G=\textrm{HNN}(H, A,t)$ where $A$ is procyclic, then $\textrm{def}(G) \geq \textrm{def}(H)$.
\end{itemize}
\end{lem}
\begin{proof}
Part (a) follows from Lemma \ref{decreasing} (a) and the obvious fact that $r(A\amalg_{C}B)\leq r(A)+r(B)+1$. For part (b) first suppose that $H=\langle X ~ | ~ R \rangle$, where $|X|=d(G)$ and $|R|=r(G)$. From the definition of HNN-extensions we have
\begin{displaymath}
G=\textrm{HNN}(H, A,t)= \langle H, t ~ | ~ tat^{-1}=f(a), \langle a \rangle =A \rangle=\langle X, t ~ | ~ R, tat^{-1}=f(a) \rangle,
\end{displaymath}
where $f: A \to G$ is a monomorphism. By Lemma 1.1 in \cite{Lubotzky}, there exists a presentation $\langle Y ~ | ~ S \rangle$ of $G$ such that $|Y|=d(G)$ and $|S|=|R|+1-(|X|+1-|Y|)$. Hence
\begin{displaymath}
\textrm{def}(G)=d(G)-r(G)\geq |Y|-|S|=|X|-|R|=\textrm{def}(H).
\end{displaymath}
\end{proof}
Now we are ready to answer positively question 9.1 in \cite{Kochloukova2}.

\begin{thm}\label{deficiency}
Let $G$ be a pro-$p$ group from the class $\mathcal{L}$. If every abelian pro-$p$ subgroup of $G$ is procyclic and $G$ itself is not procyclic, then $\textrm{def}(G)\geq 2$.
\end{thm}
\begin{proof}
Suppose that every abelian pro-$p$ subgroup of $G$ is procyclic and $G$ itself is not procyclic. Again, as in the proof of Theorem \ref{Euler}, we will use
induction on the weight $n$ of the group $G$. If $n=0$, then it is
clear that $\textrm{def}(G)\geq 2$. Let $n\geq 1$ and suppose that
any non-procyclic pro-$p$ group from the class $\mathcal{L}$ which
has weight $\leq n-1$ and in which every abelian pro-$p$ subgroup
is procyclic has deficiency $\geq 2$. By Theorem
\ref{StructureTheorem}, the group $G$ is the fundamental pro-$p$
group $\Pi_1(\G, \Gamma)$ of a finite graph of pro-$p$ groups with
infinite procyclic or trivial edge groups and finitely generated vertex groups. Moreover, each non-abelian vertex group is a
pro-$p$ group from the class $\mathcal{L}$ of weight $\leq n-1$.
Let $T_\Gamma$ be the maximal subtree of $\Gamma$, $k:=|\Gamma|$ and
$l:=|T_\Gamma|$. We can obtain $G$ by successively forming amalgamated
free products and HNN-extensions. Indeed
\begin{displaymath}
G=A_k \textrm{ where } A_l:=\G(u_1)\amalg_{\G(e_1)}\G(u_2)\amalg_{\G(e_2)}\cdots \G(u_l)\amalg_{\G(e_l)}\G(u_{l+1}),
\end{displaymath}
\begin{displaymath}
A_{l+1}:=\textrm{HNN}(A_l, \G(e_{l+1}), t_{l+1}) \textrm{ and } A_j:=\textrm{HNN}(A_{j-1}, \G(e_j), t_j) \textrm{ for } j=l+2,...,k.
\end{displaymath}
We want to show that $\textrm{def}(A_i)\geq 2$ for each $i$. Clearly, we can assume that $\G(e_i)$'s are non-trivial. Moreover, we can assume that none of the $\G(u_i)$'s is procyclic. Indeed, if $\G(u_j)\cong \mathbb{Z}_p$, then since $\G(u_j)\amalg_{\G(e_j)}\G(u_{j+1})$ is a pro-$p$ group from the class $\mathcal{L}$, we must have $\G(e_j)=\G(u_j)$ and thus $\G(u_j)\amalg_{\G(e_j)}\G(u_{j+1})=\G(u_{j+1})$. Hence, we can assume that the vertex groups $\G(u_i)$ satisfy the hypothesis of the theorem. Thus $\textrm{def}(\G(u_i))\geq 2$ for each $i$. Therefore, by Lemma \ref{def} (a), we have $\textrm{def}(A_l)\geq 2$. Moreover, Lemma \ref{def} (b) gives
\begin{displaymath}
2\leq \textrm{def}(A_l) \leq \textrm{def}(A_{l+1}) \leq \cdots \leq \textrm{def}(A_k)=\textrm{def}(G).
\end{displaymath}
\end{proof}
\smallskip

For a finitely generated pro-$p$ group $G$, denote by $s_n(G)$ the number of open subgroups of $G$ of index at most $n$. A pro-$p$ group $G$ is said to have \emph{exponential subgroup growth} if
$$ \limsup_n \frac{\textrm{log}s_n(G)}{n}>0. $$
Lackenby proved that a finitely generated pro-$p$ group $G$ has exponential subgroup growth if and only if there is a strictly descending chain $\{ G_n \}$ of open normal subgroups of $G$ such that $\displaystyle \inf_n \frac{d(G_n)-1}{|G:G_n|} > 0 $ (see \cite{Lackenby}, Theorem 8.1).

Let $G$ be a pro-$p$ group from the class $\mathcal{L}$ such that every abelian pro-$p$ subgroup of $G$ is procyclic and $G$ itself is not procyclic, and let $\{G_n \}$ be a strictly descending chain of open normal subgroups of $G$. Since $G$ is finitely presented, we have that $\chi_2(G)$ and $\chi_2(G_n)$ are well defined, where $\chi_2(G):=\sum_{i=0}^2{(-1)^i\textrm{dim}_{\mathbb{F}_p}H^i(G, \mathbb{F}_p})$ is the second partial Euler-Poincar\'e characteristic of $G$. By Lemma 3.3.15 in \cite{Wingberg} we have
$\chi_2(G_n)\leq |G:G_n|\chi_2(G)$, which implies that $\textrm{def}(G_n)-1 \geq |G:G_n| (\textrm{def}(G)-1).$ Now from the result of Lackenby mentioned above and Theorem \ref{deficiency} we have the following.
\begin{thm}
Let $G$ be a pro-$p$ group from the class $\mathcal{L}$. If every abelian pro-$p$ subgroup of $G$ is procyclic and $G$ itself is not procyclic, then $G$ has exponential subgroup growth.
\end{thm}

\section{Subgroup properties of pro-$p$ groups from the class $\mathcal{L}$}

In this section we prove parts $(5)$, $(6)$ and $(7)$ of Theorem C, stated in the introduction. We will need the following simple lemma.

\begin{lem}\label{nearnormal}
Let $G$ be a pro-$p$ group, and let $H$ and $K$ be finitely generated
subgroups of $G$. Let $A$ be a subgroup of $G$ that is contained
in both $H$ and $K$. If $A$ has finite index in both $H$ and $K$, then $A$ has a finite index subgroup that is normal
in $\langle H, K \rangle$.
\end{lem}
\begin{proof}
 Since the restrictions of the natural epimorphism $\psi: H\amalg_{A}K \to \langle H, K \rangle$ to $H$ and $K$ are injections, the amalgamated free pro-$p$ product
 $G'=H\amalg_{A}K$ is proper, i.e., $H$, $K$ and $A$ are subgroups of $G'$. If $A$ is one of $H$ or $K$, then the result is clear. Therefore we can assume that
 $A$ is different from $H$ and $K$. Note that if $U$ is an open subgroup of $A$ normal in $G'$, then $\psi(U)$ is an open subgroup of $A$ normal in $\langle H, K \rangle$. Hence in order to prove the lemma it suffices to show that $A$ has an open subgroup which is normal in $G'$.

  Since $G'=H\amalg_{A}K$ is proper, by Theorem 9.2.4 in \cite{Ribes1}, there is an indexing set $I$ and families
\begin{displaymath}
\{ U_i ~ | ~ U_i \trianglelefteq _o H \}_{i\in I} \textrm{ and }
\{ V_i ~ | ~ V_i \trianglelefteq _o K \}_{i\in I}
\end{displaymath}
with the property
\begin{displaymath}
\bigcap_{i\in I}U_i=1=\bigcap_{i\in I}V_i \textrm{ and } U_i\cap A
= V_i \cap A \textrm{ for each } i\in I.
\end{displaymath}
We can assume that these families are filtered from below. Since
$A$ is of finite index in both $H$ and $K$, it follows that there
is some $k\in I$ such that $U_k \leq A$ and $V_k\leq A$. Thus
\begin{displaymath}
U_k=U_k\cap A= V_k\cap A= V_k
\end{displaymath}
and consequently $U_k$ is an open normal subgroup of both $H$ and
$K$. Hence $U_k$ is an open subgroup of $A$ which is normal in $G$. This finishes the proof.
 \end{proof}

Let $G$ be a (profinite) group and let $H$ be a (closed) subgroup
of $G$. The \emph{commensurator} of $H$ in $G$, denoted by
$\textrm{Comm}_G(H)$, is the set
\begin{displaymath}
\{g\in G ~ | ~ H\cap gHg^{-1} \textrm{ has finite index in both }
H \textrm{ and } gHg^{-1} \}.
\end{displaymath}
It is not hard to check that $\textrm{Comm}_G(H)$ is a subgroup of
$G$ (possibly not closed if $G$ is profinite) that contains $N_G(H)$.

The following result is well known; for completeness we give its proof.
\begin{pro}\label{commensurator}
Let $G$ be a group, and let $H$ and $K$ be subgroups of $G$ such
that $K\leq H$. If $K$ has finite index in $H$, then
$\textrm{Comm}_G(K)=\textrm{Comm}_G(H)$.
\end{pro}
\begin{proof}
Let $g\in \textrm{Comm}_G(K) $. Then $K\cap gKg^{-1}$ has finite
index in $K$, and hence in $H$. Since $K\cap gKg^{-1} \subseteq
H\cap gHg^{-1}$, we have that $H\cap gHg^{-1}$ has finite index in
H. Similarly $K\cap gKg^{-1}$ has finite index in $gKg^{-1}$, and
hence $H\cap gHg^{-1}$ has finite index in $gHg^{-1}$. Thus
$\textrm{Comm}_G(K) \subseteq \textrm{Comm}_G(H)$.

Conversely, let $g\in \textrm{Comm}_G(H)$. Then $H\cap gHg^{-1}$ has finite index in $H$. Thus $K\cap H\cap gHg^{-1}=K\cap gHg^{-1}$ has finite index in $K\cap H=K$. Similarly $K\cap gKg^{-1}$ has finite index in $K\cap gHg^{-1}$. Hence $K\cap gKg^{-1}$ has finite index in $K$. In a similar way we can show that $K\cap gKg^{-1}$ has finite index in $gKg^{-1}$. Thus $\textrm{Comm}_G(H) \subseteq \textrm{Comm}_G(K)$.
\end{proof}

\begin{definition}
Let $G$ be a (pro-$p$) group and let $H$ be a
finitely generated subgroup of $G$. A \emph{root} of $H$ in $G$,
denoted by $\textrm{root}_G(H)$, is a subgroup $H'$ of $G$ that contains $H$
with $|H':H|$ finite and which contains every subgroup $K$
of $G$ that contains $H$ with $|K:H|$ finite.
\end{definition}
Note that if $H$ is a finitely generated subgroup of finite index
in $G$, then it is obvious that $\textrm{root}_G(H)=G$.

\begin{thm}\label{main}
Let $G$ be a pro-$p$ group from the class $\mathcal{L}$. Then
\begin{itemize}
\item[(1)] [Greenberg-Stallings Property] If $H$ and $K$ are finitely generated subgroups of $G$
with the property that $H\cap K$ has finite index in both $H$ and
$K$, then $H\cap K$ has finite index in $\langle H, K \rangle$; \item[(2)] If $H$ is a finitely generated subgroup of
$G$, then $H$ has a root in $G$; \item[(3)] If $H$ is a finitely
generated non-abelian subgroup of $G$, then
$|\textrm{Comm}_G(H):H|<\infty$.
\end{itemize}
\end{thm}
\begin{proof}
 (1)
 Let $H$ and $K$ be finitely generated subgroups of $G$ with the property that $H\cap K$ has finite index in both $H$ and $K$. Note that if $\langle H, K \rangle$ is abelian then the result
follows from the structure theorem of the torsion free finitely
generated abelian pro-$p$ groups (see the proof of part (2)). Thus we can assume that $\langle H, K \rangle$ is not abelian. By Lemma \ref{nearnormal}, there exists a  finitely
generated open subgroup $U$ of $H\cap K$ that is normal in $\langle H, K
\rangle$. Hence by Theorem 6.5 in \cite{Kochloukova2}, we have $|\langle H, K \rangle :
U|< \infty$. This implies that $|\langle H, K \rangle :
H\cap K|<\infty$.

\smallskip

(2) Let $H$ be an abelian finitely generated subgroup of $G$. Note
that if $H\leq A \leq G$ and $|A:H|<\infty$, then by Corollary 5.4
in \cite{Kochloukova2} it follows that $A$ is abelian. Consider
the set
\begin{displaymath}
\mathcal{S}(H)=\{ A ~ | ~ H\leq A \leq G, A \textrm{ is finitely
generated and abelian} \}.
\end{displaymath}
Let
$A_1\leq A_2\leq \cdots$ be an ascending chain of elements in
$\mathcal{S}(H)$. Then $A=\langle \cup_{i\geq 1}A_i
\rangle$ is abelian. Using Corollary 5.5 in \cite{Kochloukova2} and obvious induction, it is not hard to see that $A$ is finitely generated.  Thus
every ascending chain in $\mathcal{S}(H)$ has an upper bound. By
Zorn's lemma it follows that $\mathcal{S}(H)$ has a maximal
element; denote this element by $S$. From the structure theorem of
finitely generated free modules over principal ideal domains it
follows that there exists a basis $y_1, y_2,..., y_n$ of $S$ so
that $p^{a_1}y_1, p^{a_2}y_2,...,p^{a_m}y_m$ is a basis of $H$
where $m\leq n$ and $a_1, a_2,...,a_m$ are non-zero integers with
the relation $a_1\leq a_2 \leq\cdots \leq a_m$. Set $N=\langle
y_1,...,y_m \rangle$; it is easy to see that
$N=\textrm{root}_G(H)$.

Now let $H$ be a non-abelian finitely generated subgroup of $G$. By Theorem \ref{Euler} we have that $\chi(H)<0$. If $H\leq K$ and $|K:H|< \infty$, then from the multiplicativity of the Euler-Poincar\'e characteristic it follows that $\chi(H)\leq \chi(K)=\frac{\chi(H)}{|K:H|}<0$. Choose $K$ such that $H\leq K$, the index $|K:H|< \infty$ and $\chi(K)$ is as large as possible. We claim that $K$ is a root of $G$. Indeed, suppose that there is some $M\leq G$ such that $H\leq M$, the index $|M:H|<\infty$ and $K$ does not contain $M$. Then by Greenberg-Stallings property, we have that $H$ is also of finite index in $A=\langle K, M \rangle$. But then $\chi(A)=\frac{\chi(K)}{|A:K|}>\chi(K)$, which is a contradiction. Thus we must have $K=\textrm{root}_G(H)$.
\smallskip

(3) Let $H$ be a finitely generated non-abelian subgroup of $G$. By (2), $H$ has a root in $G$. By Proposition
\ref{commensurator} we have
\begin{displaymath}
\textrm{Comm}_G(H)=\textrm{Comm}_G(\textrm{root}_G(H)).
\end{displaymath}
Since $\textrm{root}_G(\textrm{root}_G(H))=\textrm{root}_G(H)$, it
suffices to prove that if $H=\textrm{root}_G(H)$, then
$H=N_G(H)=\textrm{Comm}_G(H)$.

Suppose that $H=\textrm{root}_G(H)$. By Theorem 6.7 in \cite{Kochloukova2},  $H$ has finite index in
$N_G(H)$. Hence we have
\begin{displaymath}
H\leq N_G(H)\leq \textrm{root}_G(H)=H.
\end{displaymath}
Thus $H=N_G(H)$. Also, it is clear that $N_G(H)\leq
\textrm{Comm}_G(H)$. It remains to show that $\textrm{Comm}_G(H)
\leq N_G(H)$. Let $g\in \textrm{Comm}_G(H)$. This means that
$H\cap gHg^{-1}$ has finite index in both $H$ and $gHg^{-1}$, and
as a consequence we have
\begin{displaymath}
\textrm{root}_G( H\cap
gHg^{-1})=\textrm{root}_G(gHg^{-1})=\textrm{root}_G(H)=H.
\end{displaymath}
It follows that
\begin{displaymath}
\langle gHg^{-1}, H \rangle = H.
\end{displaymath}
Suppose that $g\in \textrm{Comm}_G(H)\backslash N_G(H)$. Then
$gHg^{-1}\neq H$, and hence $H$ is properly contained in $\langle
gHg^{-1}, H \rangle = H$, a contradiction. Thus we must have
$\textrm{Comm}_G(H)\backslash N_G(H)=\emptyset$, i.e.,
$N_G(H)=\textrm{Comm}_G(H)$, as desired. This finishes the proof.
\end{proof}

\begin{definition}
For a given subgroup $H$ of $G$, the \emph{normalizer tower} of $H$ in $G$ is defined as
\begin{displaymath}
N^0_G(H)=H, ~ N^{\alpha + 1}_G(H)=N_G(N^{\alpha}_G(H))
\end{displaymath}
and if $\alpha$ is a limit ordinal, then
\begin{displaymath}
N^{\alpha}_G(H) = \bigcup_{\beta < \alpha} N^{\beta}_G(H).
\end{displaymath}
\end{definition}

By part (2) of the above theorem and Theorem 6.7 in \cite{Kochloukova2} we have the following.
\begin{cor}
Let $G$ be a pro-$p$ group from the class $\mathcal{L}$. If $H$ is a finitely generated non-abelian subgroup of $G$, then the normalizer tower of $H$ in $G$ stabilizes after finitely many steps, i.e., it has finite length.
\end{cor}

\begin{pro}\label{comindex}
Let $G$ be a pro-$p$ group from the class $\mathcal{L}$ and let $H$ be a non-abelian finitely generated subgroup of $G$.
Then $\textrm{Comm}_G(H)=\textrm{root}_G(H)$. In particular, the group $H$ has finite index in $\textrm{Comm}_G(H)$.
\end{pro}
\begin{proof} Was performed in the proof of part (3) of Theorem \ref{main}.

\end{proof}

The following result generalizes Corollary 6.6 in \cite{Kochloukova2}.
\begin{cor}\label{cor1}
Let $G$ be a pro-$p$ group from the class $\mathcal{L}$. If $F$ is a finitely generated free pro-$p$ subgroup of $G$ with $d(F)$ not congruent to $1$ modulo $p$, then
 \begin{displaymath}
 F=N_G(F)=\textrm{root}_G(F)=\textrm{Comm}_G(F).
 \end{displaymath}
\end{cor}
\begin{proof}
Let $F$ be a finitely generated subgroup of $G$ with $d(F)$ not
congruent to $1$ modulo $p$. By part (2) of Theorem \ref{main} we know
that $F$ has a root. Suppose that $\textrm{root}_G(F)\neq F$. Then
there is a subgroup $H$ of $G$ that contains $F$ and such that
$|H:F|=p$. Since $H$ is torsion free, by Serre's result
\cite{Serre} we have that $H$ is a free pro-$p$ group. From
Nielsen-Schreier formula we have $d(F)-1=p(d(H)-1)$. Thus
$d(F)\equiv 1 (\textrm{mod }p)$, which is a contradiction. Thus we
must have $F=\textrm{root}_G(F)$. By Theorem 6.7 in \cite{Kochloukova2},  $F$ has finite index in
$N_G(F)$. Hence $F\leq N_G(F)\leq \textrm{root}_G(F)=F $. By the
previous proposition we have $F=N_G(F)=\textrm{root}_G(F)=\textrm{Comm}_G(F)$.
\end{proof}
A finitely generated subgroup $H$ of a group $G$ is said to be \emph{self-rooted} if it  has a root in $G$ and $\textrm{root}_G(H)=H$. From the above corollary it follows that if $G$ is a non-abelian pro-$p$ group from the class $\mathcal{L}$, then for any $n\in \mathbb{N}$ there is a self-rooted finitely generated subgroup $F$ of $G$ with $d(F)>n$.

\smallskip

Let $G$ be a pro-$p$ group from the class $\mathcal{L}$. To every
finitely generated self-rooted subgroup $H$ of $G$ we associate
the set
\begin{displaymath}
H^*=\{ U ~ | ~ U\leq H \textrm{ and } |H:U| < \infty   \}.
\end{displaymath}
Consider the sets
\begin{displaymath}
\mathcal{M}(G)= \{ H ~ | ~ H \textrm{ is a finitely generated subgroup of } G \},
\end{displaymath}
\begin{displaymath}
 \overline{G}= \{ H ~ | ~ H\leq G \},
 \end{displaymath}
 \begin{displaymath}
\mathcal{L}(G)=\{H^* ~ | ~ H \textrm{ is a finitely generated self-rooted subgroup of } G  \}
\end{displaymath}
 and recall that we can consider $\overline{G}$ as a lattice with the standard meet and joint operations for groups. One can easily prove the following result.
 \begin{pro} \label{lattice}
  Let $G$ be a pro-$p$ group from the class $\mathcal{L}$.
\begin{itemize}
\item[a)] If $H$ is a self-rooted subgroup of $G$, then $H^*$ is a convex sublattice of $\overline{G}$ with greatest element $H$ and without a least element.
\item[b)] The set $\mathcal{L}(G)$ forms a partition of $\mathcal{M}(G)$, i.e., any two distinct elements in $\mathcal{L}(G)$ are disjoint and $\mathcal{M}(G)$ is equal to the union of all the elements in $\mathcal{L}(G)$.
\end{itemize}
\end{pro}

\section{Demushkin groups}

\begin{definition}
Let $G$ be a pro-$p$ group. We say $G$ is an \emph{$IF$-group} if all finitely generated infinite index subgroups of $G$ are free pro-$p$ groups.
\end{definition}
Free pro-$p$ groups are obvious examples of $IF$-groups. Infinite Demushkin groups, whose definition we recall below, form another family of $IF$-groups.
\begin{definition} A pro-$p$ group $G$ is called a $Demushkin$ group if it satisfies the following conditions:
\begin{itemize}
\item[(i)] $\textrm{dim} _{\mathbb{F}_p} H^1(G,\mathbb{F}_p) < \infty$,
\item[(ii)] $\textrm{dim} _{\mathbb{F}_p} H^2(G,\mathbb{F}_p)=1$, \textrm{ and }
\item[(iii)] the cup-product $H^1(G,\mathbb{F}_p)\times H^1(G,\mathbb{F}_p) \to H^2(G,\mathbb{F}_p) \cong \mathbb{F}_p$ is a non-degenerate bilinear form.
\end{itemize}
\end{definition}
Infinite Demushkin groups are precisely the Poincar\'e duality groups of dimension $2$. It is well known that if $G$ is an infinite Demushkin group, then all finite index subgroups of $G$ are Demushkin and all infinite index subgroups of $G$ are free pro-$p$ (see \cite{Serre1}).

Next we discuss some other examples of $IF$-groups, which similarly as Demushkin groups, appear in number theory. Let $K$ be a discrete valuation field with perfect residue field $k$ of characteristic $p>0$, and let $K_{\textrm{sep}}$ be a separable closure of $K$. Denote by $K(p)$ the maximal $p$-extension of $K$ in $K_{\textrm{sep}}$ and let $\Gamma(p)=Gal(K(p)/K)$. When $v_0>-1$ and $\Gamma(p)^{(v_0)}$ is the ramification subgroup of $\Gamma(p)$ in upper numbering (cf. \cite{Serre2}, Ch. III), Abrashkin proved that any closed but not open finitely generated subgroup of the quotient $\Gamma(p)/\Gamma(p)^{(v_0)}$ is a free pro-$p$ group \cite{Abrashkin}. Hence, according to our definition, it is an $IF$-group. If $-1<v_0\leq 1$, then $\Gamma(p)/\Gamma(p)^{(v_0)}$ coincides with the Galois group of the maximal $p$-extension of the residue field $k$, and thus it is a free pro-$p$ group. If $v_0>1$ then $\Gamma(p)/\Gamma(p)^{(v_0)}$ is far from being a free pro-$p$ group. If $k$ is infinite, then it is not finitely generated, and if $k$ is finite, then it is finitely generated but the number of its relations is infinite (cf. \cite{Gordeev}). Thus if $v_0>1$, then $\Gamma(p)/\Gamma(p)^{(v_0)}$ is an $IF$-group which is neither free pro-$p$ nor Demushkin.

\begin{thm}\label{Demushkinnormalizer}
Let $G$ be a finitely presented $IF$-group with an open subgroup of deficiency greater than 1. If $H$ is a finitely generated subgroup of $G$ that contains a non-trivial normal subgroup of $G$, then $H$ has finite index in $G$.
\end{thm}
\begin{proof}
Let $H$ be a finitely generated subgroup of $G$ that contains a non-trivial normal subgroup $K$ of $G$. Let $G_1$ be an open subgroup of $G$ such that $\textrm{def}(G_1)\geq 2$. Then $H\cap G_1$ is a finitely generated subgroup of $G_1$ and $K\cap G_1$ is a non-trivial normal subgroup of $G_1$ that is contained in $H\cap G_1$. Note that  $H\cap G_1$ has finite index in $G_1$ if and only if $H$ has finite index in $G$. Thus, without loss of generality, we can assume that $\textrm{def}(G)\geq 2$.

Now we use an idea of D. Kochloukova in \cite{Kochloukova3}, where the theorem is proved for Demushkin groups with $\chi(G)\neq 0$.
Suppose that $|G:H|=\infty$. Firstly we show that $H$ cannot be procyclic. Indeed, if $H\cong \mathbb{Z}_p$, then we must have $K\cong \mathbb{Z}_p$, and by Theorem 3 in \cite{Hillman}, it follows that $\textrm{def}(G)\leq 1$. Thus $H$ is a non-abelian free pro-$p$ group.
Note that $\chi(H)=1-d(H)\leq {-1}$ and consider the set
\begin{displaymath}
\mathcal{T}(H)=\{U ~ | ~ H\leq U \leq G, U \textrm{ is finitely generated } \textrm{ and } |G:U|=\infty \}.
\end{displaymath}
Let $M \in \mathcal{T}(H)$. Then $M$ is a finitely generated non-abelian free pro-$p$ group and $H$ is a finitely generated subgroup of $M$ that contains the normal subgroup $K$ of $G$. By Proposition 3.3 in \cite{LubotzkyA1}, it follows that $|M:H| < \infty$. Thus $M$ is finitely generated and $\chi(M)=1-d(M)\leq {-1}$. From the multiplicativity of the Euler-Poincar\'e characteristic on finite index subgroups we have
\begin{displaymath}
\chi(H)=|M:H|\chi(M).
\end{displaymath}
Since $-\chi(M)\geq 1$, it follows that
\begin{displaymath}
|M:H|=\frac{\chi(H)}{\chi(M)}\leq -\chi(H).
\end{displaymath}
Thus there is an upper bound for $|M:H|$. This implies that every ascending chain of elements in $\mathcal{T}(H)$ has an upper bound. By Zorn's lemma it follows that $\mathcal{T}(H)$ has a maximal element.

 Let $N$ be a maximal element of  $\mathcal{T}(H)$. Since $N$ is a closed subgroup of $G$, we have
\begin{displaymath}
N=\cap\{V ~ | ~ N \leq V \leq_o G   \}.
\end{displaymath}
Moreover, since $N$ has infinite index in $G$, there is a sequence $V_1\geq V_2 \geq\cdots V_1\geq V_{i+1} \geq \cdots $ of open subgroups in $G$ such that $N=\cap_{i\geq 1} V_i$.

For each $i\geq 1$, choose $w_i \in V_i \backslash N$ and set $W_i=\langle N, w_i \rangle$. Then we have
$N=\cap_{i\geq 1} W_i$.
Note that there is no $i\geq 1$ with $|G: W_i|=\infty$, because otherwise it would contradict the maximality of $N$ in $\mathcal{T}(H)$. Hence $|G: W_i|<\infty$ for all $i\geq 1$. Since $G$ is finitely presented, we have that $\chi_2(G)$ and $\chi_2(W_i)$ are well defined, where $\chi_2(G)$ is the second partial Euler-Poincar\'e characteristic of $G$. By Lemma 3.3.15 in \cite{Wingberg} we have $\chi_2(W_i)\leq |G:W_i|\chi_2(G)$, which implies that $\textrm{def}(W_i)-1 \geq |G:W_i| (\textrm{def}(G)-1).$
Thus we have
\begin{displaymath}
|G:W_i|\leq \frac{\textrm{def}(W_i)-1}{\textrm{def}(G)-1}\leq \frac{d(W_i)}{\textrm{def}(G)-1}\leq \frac{d(N)+1}{\textrm{def}(G)-1}.
\end{displaymath}
Hence the index $|G:W_i|$ has an upper bound that does not depend on $i$. This implies that there are only finitely many possibilities for $W_i$. Hence, $N=\cap_{i\geq 1} W_i $ has finite index in $G$, a contradiction. This finishes the proof.
\end{proof}
\begin{rem*}
Note that the result in the above theorem in general is not valid if we do not assume that $G$ has an open subgroup of deficiency greater than 1. For instance, if $G$ is an infinite solvable Demushkin group, then every open subgroup of $G$ has deficiency $1$ and $G$ has a normal subgroup of infinite index isomorphic to $\mathbb{Z}_p$.
\end{rem*}
As an immediate consequence of Theorem \ref{Demushkinnormalizer} we get the following.
\begin{cor}
Let $G$ be a finitely presented $IF$-group with an open subgroup of deficiency greater than 1. If $H$ is a non-trivial finitely generated normal subgroup of $G$, then $H$ has finite index in $G$.
\end{cor}
\begin{rem*}
The above corollary is just a special case of Theorem 3 in \cite{Hillman}. Moreover, note that the above result is true for any free pro-$p$ group, not only for finitely generated free pro-$p$ groups. Indeed, if $G$ is a free pro-$p$ group and $H$ is a finitely generated subgroup of $G$ of infinite index, then one can easily find a finitely generated subgroup $G'$ of $G$ such that $H\leq G'$ and $H$ has infinite index in $G'$; this is impossible by the above corollary.
\end{rem*}
\begin{cor}\label{DemNorm}
Let $G$ be a finitely presented pro-$p$ group with an open subgroup of deficiency greater than 1. Suppose that all infinite index subgroups of $G$ are free pro-$p$ groups. Then any non-trivial finitely generated subgroup $H$ of $G$ has finite index in its normalizer in $G$.
\end{cor}
\begin{proof}
Let $H$ be a finitely generated subgroup of $G$. If $N_G(H)$ is of infinite index in $G$, then it is a free pro-$p$ group. Since $H$ is a finitely generated normal subgroup of the free pro-$p$ group $N_G(H)$, by the above remark, it must be of finite index in $N_G(H)$. Now suppose that $N_G(H)$ has finite index in $G$. Let $K$ be an open subgroup of $G$ such that $\textrm{def}(K)\geq 2$. Then $K\cap N_G(H)$ is an open subgroup of $K$ and
\begin{displaymath}
\textrm{def}(K\cap N_G(H))-1 \geq |K:K\cap N_G(H)|(\textrm{def}(K)-1)\geq 1.
\end{displaymath}
Since $K\cap N_G(H)$ has finite index in $N_G(H)$, one can apply the above corollary and obtain that $H$ has finite index in $N_G(H)$.
\end{proof}
Note that Corollary \ref{DemNorm} is an analogue of Theorem 6.7 in \cite{Kochloukova2}.

\begin{thm}\label{mainDem}
Let $G$ be a finitely presented $IF$-group with an open subgroup of deficiency greater than 1. Then
\begin{itemize}
\item[(1)] [Greenberg-Stallings Property] If $H$ and $K$ are finitely generated subgroups of $G$
with the property that $H\cap K$ has finite index in both $H$ and
$K$, then $H\cap K$ has finite index in $\langle H, K \rangle$; \item[(2)] If $H$ is a finitely generated subgroup of
$G$, then $H$ has a root in $G$; \item[(3)] Suppose in addition that all infinite index subgroups of $G$ are free pro-$p$ groups. Then $|\textrm{Comm}_G(H):H|<\infty$ for any non-trivial finitely generated subgroup $H$ of $G$.
\end{itemize}
\end{thm}
\begin{proof}
(1)
 Let $H$ and $K$ be finitely generated subgroups of $G$ with the property that $H\cap K$ has finite index in both $H$ and $K$. Note that if $\langle H, K \rangle$ is abelian then the result
follows from the structure theorem of the torsion free finitely
generated abelian pro-$p$ groups. Thus we can assume that $\langle H, K \rangle$ is not abelian. By Lemma \ref{nearnormal}, there exists a  finitely
generated open subgroup $U$ of $H\cap K$ that is normal in $\langle H, K
\rangle$. Hence, by Theorem \ref{Demushkinnormalizer}, we have $|\langle H, K \rangle :
U|< \infty$. This implies that $|\langle H, K \rangle :
H\cap K|<\infty$.

\smallskip

(2)
Let $H$ be a finitely generated subgroup of $G$. If $H$ has finite index in $G$, then $\textrm{root}_G(H)=G$. Therefore, we may assume that $H$ has infinite index in $G$. Consider the set
\begin{displaymath}
\mathcal{R}(H)=\{ K ~ | ~ H \leq K \leq G \textrm{ and } |K:H|< \infty \}.
\end{displaymath}
It is easy to see that $H$ has a root in $G$ if and only if the greatest element of $\mathcal{R}(H)$ exists. Thus it suffices to show the existence of the greatest element of $\mathcal{R}(H)$.

Firstly we consider the case when $H$ is not procyclic. Since $H$ is a finitely generated non-abelian free pro-$p$ group, we have $\chi(H)=1-d(H)\leq {-1}$. Let $K\in \mathcal{R}(H)$. Since $|K:H| < \infty$, it follows that $K$ is also a finitely generated non-abelian free pro-$p$ group, and $\chi(K)=1-d(K)\leq {-1}$. From the multiplicativity of the Euler-Poincar\'e characteristic on finite index subgroups we have
\begin{displaymath}
\chi(H)=|K:H|\chi(K).
\end{displaymath}
Since $-\chi(K)\geq 1$, it follows that
\begin{displaymath}
|K:H|=\frac{\chi(H)}{\chi(K)}\leq -\chi(H).
\end{displaymath}
Thus there is an upper bound for $|K:H|$. This implies that every ascending chain of elements in $\mathcal{R}(H)$ has an upper bound. By Zorn's lemma it follows that $\mathcal{R}(H)$ has a maximal element.

Next, suppose that $H\cong \mathbb{Z}_p$ and let $H_1\leq H_2 \leq \cdots$ be an ascending chain of elements in $\mathcal{R}(H)$. Let $L=\langle \cup_{i\geq 1}H_i \rangle$. Then $L$ is a closed abelian subgroup of $G$, so we must have $L\cong \mathbb{Z}_p$ (because $G$ is a finitely presented $IF$-group with an open subgroup of deficiency greater than 1). Since the only closed subgroup of infinite index in $\mathbb{Z}_p$ is the trivial one, we have $|L:H|<\infty$.
 Thus every ascending chain in $\mathcal{R}(H)$ has an upper bound. By Zorn's lemma it follows that $\mathcal{R}(H)$ has a maximal element.

 Now let $N$ be a maximal element of $\mathcal{R}(H)$. We claim that $N$ is the greatest element of $\mathcal{R}(H)$. Suppose this is not true. Then there is $A \in \mathcal{R}(H)$ such that $A\nleq N$, and so, by the Greenberg-Stallings property we have $|\langle N, A \rangle : H | < \infty$. Thus $\langle N, A \rangle$ is an element of $\mathcal{R}(H)$ which properly contains $N$. Hence $N$ is not a maximal element of $\mathcal{R}(H)$, which is a contradiction.

 \smallskip

 (3) Let $H$ be a finitely generated subgroup of $G$. Then $|N_G(H):H| < \infty$, by Corollary \ref{DemNorm}. Now if we proceed as in the proof of part (3) of Theorem \ref{main}, we get $|\textrm{Comm}_G(H):H|<\infty$.
 \end{proof}

 By part (2) of the above theorem and Corollary \ref{DemNorm} we have the following.
\begin{cor}
Let $G$ be a finitely presented pro-$p$ group with an open subgroup of deficiency greater than 1. Suppose that all infinite index subgroups of $G$ are free pro-$p$ groups.  Then the normalizer tower in $G$ of any non-trivial finitely generated subgroup  $H$ of $G$ stabilizes after finitely many steps, i.e., it has finite length.
\end{cor}

\begin{pro}
Let $G$ be a finitely presented pro-$p$ group with an open subgroup of deficiency greater than 1. Suppose that all infinite index subgroups of $G$ are free pro-$p$ groups. Then for a non-trivial finitely generated subgroup $H$ of $G$ we have $\textrm{Comm}_G(H)=\textrm{root}_G(H)$. In particular, the group $H$ has finite index in $\textrm{Comm}_G(H)$.
\end{pro}
\begin{proof}
Same as the proof of Proposition \ref{comindex}.
\end{proof}
The following result is an analogue of Corollary \ref{cor1}.
\begin{cor}
Let $G$ be a finitely presented $IF$-group with an open subgroup of deficiency greater than 1. Then for any non-trivial finitely generated subgroup $H$ of $G$ with $d(H)$ not congruent to $1$ modulo $p$ we have
 \begin{displaymath}
 H=\textrm{root}_G(H).
 \end{displaymath}
If in addition we suppose that all infinite index subgroups of $G$ are free pro-$p$ groups, then
 \begin{displaymath}
 H=N_G(H)=\textrm{root}_G(H)=\textrm{Comm}_G(H).
 \end{displaymath}
\end{cor}
\begin{proof}
Same as the proof of Corollary \ref{cor1}.
\end{proof}
Recall that infinite Demushkin groups have positive deficiency. Moreover, if $G$ is an infinite Demushkin group then it is solvable if and only if $\textrm{def}(G)=1$. Thus non-solvable Demushkin groups have deficiency greater than $1$; hence all the results stated in this section hold for non-solvable Demushkin groups.

\smallskip

Let $\overline{G}$, $\mathcal{M}(G)$, $H^*$ and $\mathcal{L}(G)$ be defined as in Section 4. We have the following.
\begin{pro}\label{lattice1}
Let $G$ be a finitely presented $IF$-group with an open subgroup of deficiency greater than 1.
\begin{itemize}
\item[a)] If $H$ is a self-rooted subgroup of $G$, then $H^*$ is a convex sublattice of $\overline{G}$ with greatest element $H$ and without a least element.
\item[b)] The set $\mathcal{L}(G)$ forms a partition of $\mathcal{M}(G)$, i.e., any two distinct elements in $\mathcal{L}(G)$ are disjoint and $\mathcal{M}(G)$ is equal to the union of all the elements in $\mathcal{L}(G)$.
\end{itemize}
\end{pro}

\section{Abstract limit groups}
In \cite{Rosset}, as we mentioned in the introduction, Rosset proved that every finitely generated subgroup of a free group has a root. The following theorem generalizes this result to the class of abstract limit groups.

\begin{thm}\label{last}
Let $G$ be an abstract limit group. If $H$ is a finitely generated subgroup of $G$, then $H$ has a root in $G$.
\end{thm}
\begin{proof}
 We only need to mention that by Theorem 6 in \cite{Serbin}, abstract limit groups satisfy the Greenberg-Stallings property and by Lemma 5 in \cite{Kochloukova1}, non-abelian abstract limit groups have negative Euler characteristic. The rest of the proof is similar to the proof of part (2) of Theorem \ref{main}.
\end{proof}

By the above theorem  and Theorem 1 in \cite{Bridson1} we have the following.

\begin{cor}
Let $G$ be an abstract limit group. If $H$ is a finitely generated non-abelian subgroup of $G$, then the normalizer tower of $H$ in $G$ stabilizes after finitely many steps, i.e., it has finite length.
\end{cor}

Finally, let us note that the result of Proposition \ref{lattice} also holds for abstract limit groups.

\bigskip

\ackn This work was carried out while the first author was holding a CNPq Postdoctoral Fellowship at the University of Bras\'{\i}lia. He would like to thank CNPq for the financial support and the Department of Mathematics at the University of Bras\'{\i}lia for its warm hospitality and the excellent research environment.

\bibliographystyle{plain}

\end{document}